\documentclass[11pt]{article}

\usepackage{amsmath,amssymb,amsthm,amsfonts,mathtools}
\usepackage{mathrsfs}
\usepackage{bm}
\usepackage{bbm}
\usepackage{eucal}
\usepackage{graphicx}
\usepackage{float}
\usepackage{color}
\usepackage{subfigure}
\usepackage{indentfirst}

\usepackage{cite}
\usepackage{hyperref}
\hypersetup{
    colorlinks=true,
    linkcolor=blue,
    citecolor=blue,
    urlcolor=blue
}

\allowdisplaybreaks

\theoremstyle{plain}
\newtheorem{theorem}{Theorem}[section]
\newtheorem{lemma}[theorem]{Lemma}
\newtheorem{proposition}[theorem]{Proposition}

\newtheorem{definition}[theorem]{Definition}
\newtheorem{remark}[theorem]{Remark}

\newtheorem{Assumption}{H\!\!} 

\numberwithin{equation}{section}

\newcommand{\be}{\begin{eqnarray}}
\newcommand{\ee}{\end{eqnarray}}
\newcommand{\by}{\begin{eqnarray*}}
\newcommand{\ey}{\end{eqnarray*}}
\newcommand{\bn}{\begin{enumerate}}
\newcommand{\en}{\end{enumerate}}

\begin{document}
\date{}
\title{\bf Smoluchowski--Kramers approximation with L\'evy noise in the Meyer--Zheng topology  \footnote{This work was supported by
the National Natural Science Foundation of China  NSFC No.12371243.
}}
\author{Xueru Liu\footnote{327625941@qq.com}
\hskip1cm Qianqian Jiang \footnote{jqq172515@gmail.com}
 \hskip1cm Wei Wang\footnote{Corresponding author, wangweinju@nju.edu.cn} \\
\texttt{{\scriptsize School of Mathematical Sciences, Dalian University of Technology, 116024 Dalian, P. R. China}}\\
\texttt{\scriptsize  Department of Mathematics, Ruhr University Bochum, D-44801 Bochum, Germany}\\
\texttt{{\scriptsize School of Mathematics, Nanjing University,
Nanjing, 210023, P. R. China}}\\
}\maketitle

\begin{abstract}
We study the Smoluchowski--Kramers approximation for a stochastic wave
equation with state-dependent damping on a bounded domain, driven by both
a \(Q\)-Wiener process and a L\'evy process with finite second moment. 
As \(\varepsilon\to0\), we prove that \(u^\varepsilon\) converges in
distribution, in the Meyer--Zheng topology, to the unique weak solution
of an overdamped stochastic parabolic equation.
The proof relies on a nonlinear transformation associated with the damping
coefficient, uniform energy estimates, and compactness arguments in the
pseudo-path topology. 
We identify the limiting equation, which contains both
the classical Gaussian noise-induced drift caused by state-dependent damping
and an explicit jump correction generated by the L\'evy noise. We show that
the latter coincides exactly with the Marcus-to-It\^o correction associated
with the canonical jump flow induced by the nonlinear damping transformation.
 
\end{abstract}
\maketitle
\maketitle


\section{Introduction}

In this paper, we investigate the small-mass limit $\epsilon \to 0$ of a L\'evy noise–driven stochastic wave equation with state-dependent damping, given by
\begin{equation}\label{e:SLE-11}
\left\{
\begin{aligned}
du^\epsilon(t) &=v^\epsilon(t)\,dt,\\
\epsilon\,dv^\epsilon(t) &=\Delta u^\epsilon(t)\,dt-\gamma(u^\epsilon(t))v^\epsilon(t)\,dt+F(u^\epsilon(t))\,dt+\sigma(u^\epsilon(t))\,dW^Q(t)\\
&\quad +\int_{Z_1}G(u^\epsilon(t-),z)\,\tilde{\mathcal N}(dt,dz) +\int_{\mathbb R^m\setminus Z_1} G(u^\epsilon(t-),z)\,\mathcal N(dt,dz),\\
u^\epsilon(t)|_{\partial\mathcal O}&=0,\qquad u^\epsilon(0)=u_0,\qquad v^\epsilon(0)=v_0 ,
\end{aligned}
\right.
\end{equation}
where $\mathcal{O} \subset \mathbb{R}^d$ is a bounded domain and
$\Delta$ denotes the Laplacian with homogeneous Dirichlet boundary conditions. The dynamics take place in the Hilbert space $H^1_0(\mathcal{O}) \times L^2(\mathcal{O})$.

Equation~\eqref{e:SLE-11} may be interpreted as a damped Hamiltonian system subject to L\'evy-type random forcing. The Smoluchowski--Kramers approximation, corresponding to the singular perturbation limit $\epsilon \to 0$, replaces the second-order dynamics with a first-order effective model. While this classical approximation is well-understood in finite dimensions with constant damping, the presence of state-dependent friction and jump noise introduces subtle corrections in the limiting behavior.

The Smoluchowski--Kramers approximation has a long history, dating back to the seminal works of Smoluchowski~\cite{Smoluchowski} and Kramers~\cite{Kramers}.
In finite dimensions, small-mass limits for stochastic differential equations have been widely studied under various conditions, see, for example,
\cite{MF, KS} for constant friction.  When the damping coefficient depends on the state, it is known to generate
nontrivial noise-induced drift terms in the limit, see, for instance~\cite{MWHu,Herzog,shamgv:smklim}.

Recently, the Smoluchowski--Kramers approximation has been extended to infinite-dimensional systems, particularly stochastic partial differential
equations (SPDEs),see~\cite{SMM, SCMF, SM, rmt, L1, WWSCG}. Most of these works assume a constant damping coefficient, which leads to
limiting equations without additional correction terms. However, as already observed in finite-dimensional settings
\cite{Herzog, shamgv:smklim}, state-dependent damping interacts nontrivially with stochastic forcing and can induce effective drift terms in the limiting
dynamics.

In this context, Cerrai and Xi~\cite{SCGX} studied the small-mass limit for stochastic wave equations with state-dependent damping and multiplicative
Gaussian noise. They proved convergence to a quasilinear stochastic heat equation containing a noise-induced drift correction.
Related questions concerning stationary solutions under state-dependent friction were  investigated in~\cite{MXie}.
Subsequently, Cerrai and Debussche~\cite{SA} extended the analysis to systems of wave equations with vector-valued damping and noise coefficients.
The convergence of stationary solutions, long-time statistics, and related
averaging problems has also been studied in
\cite{Glatt-Holtz,Nguyen,WWjsp}.

All the aforementioned infinite-dimensional results focus on Gaussian noise. However, stochastic forcing arising in applications often includes non-Gaussian
components, such as jump discontinuities or heavy-tailed events, see~\cite{DD, DVYAV, DJQ, WAW} for applications.
This motivates the study of the Smoluchowski--Kramers approximation in the
presence of  L\'evy noise.

For systems with constant damping, small-mass limits under  L\'evy noise were
established in~\cite{zhangs}, while related results for interacting particle
systems appear in~\cite{wangzibo2}. However, to the best of our knowledge, the combined effect of  L\'evy noise and
state-dependent damping in infinite-dimensional systems has not been addressed.

\noindent\textbf{Main result.}
The main result of this paper is the rigorous derivation of the small-mass limit of the second-order stochastic system~\eqref{e:SLE-11} as
\(\epsilon\to0\). A key difficulty comes from the presence of L\'evy jumps, the limiting process may have jumps, while \(u^\epsilon\) has continuous
trajectories. Therefore, the usual Skorokhod \(J_1\) topology is too strong for the compactness argument.

To overcome this difficulty, we work in the Meyer--Zheng topology, also known as the pseudo-path topology. After introducing the nonlinear
transformation
\[
\eta^\epsilon:=g(u^\epsilon), \qquad g(r)=\int_0^r \gamma(a)\,da,
\]
we prove the tightness of \(\left\{\mathcal L(\eta^\epsilon)\right\}_{\epsilon>0}\) in \(D([0,T];H^{-1})\) endowed with the Meyer--Zheng topology. This topology
is weak enough to capture the convergence of the time-averaged paths and is therefore well suited to the jump-driven small-mass limit. Combined with
uniform energy estimates and the identification of all limit points, this yields the convergence
\[
\eta^\epsilon \Longrightarrow\eta\qquad\text{in }D([0,T];H^{-1})_{\rm MZ}.
\]
Consequently,
\[
u^\epsilon=g^{-1}(\eta^\epsilon)\Longrightarrow u:=g^{-1}(\eta) \qquad \text{in }D([0,T];H^{-1})_{\rm MZ}.
\]
The limiting dynamics is a stochastic parabolic equation with two additional
correction terms. More precisely, the limiting equation is given by
\begin{eqnarray}\label{LAST}
u(t) &=& u_0
+ \int_0^t \frac{1}{\gamma(u(s))}\,\Delta u(s)\,ds
+ \int_0^t \frac{1}{\gamma(u(s))}\,F(u(s))\,ds
+ \int_0^t \frac{1}{\gamma(u(s))}\,\sigma(u(s))\,dW^Q(s) \nonumber\\
&&+ \int_0^t\int_{\mathbb R^m}
\frac{1}{\gamma(u(s-))}\,G(u(s-),z)\,\tilde{\mathcal N}(ds,dz)
- \int_0^t
\frac{\gamma'(u(s))}{2\,\gamma(u(s))^3}
\sigma(u(s))\sigma^\ast(u(s))\,ds
\nonumber\\
&&+ \sum_{0<s\le t}
\left[
u(s)-u(s-)
-\frac{1}{\gamma(u(s-))}
\int_{u(s-)}^{u(s)} \gamma(r)\,dr
\right].
\end{eqnarray}

\medskip
\noindent\textbf{Marcus integral structure.}
A notable feature of~\eqref{LAST} is its consistency with the Marcus formulation of jump-driven dynamics. The additional drift term
\[
\sum_{0\le s\le t}
\left[
u(s)-u(s-)
-\frac{1}{\gamma(u(s-))}
\int_{u(s-)}^{u(s)} \gamma(r)\,dr
\right]
\]
coincides precisely with the It\^o correction arising from the Marcus-to-It\^o conversion under the transformation $u = g^{-1}(\eta)$, where $g(u) = \int_0^u \gamma(r)\,dr$.  
 This confirms that the small-mass limit preserves the pathwise structure of the original jump dynamics and admits a natural interpretation in terms of Marcus integrals,  see Section~\ref{subsec:marcus-drift}  for a detailed derivation.

\medskip
\noindent\textbf{Organization of the paper.}
Section~\ref{sec:pre} introduces the standing assumptions, the functional analytic framework, and the precise notion of weak solution. It also establishes the existence and uniqueness of mild solutions to equation~\eqref{e:SLE-11}.  
Section~\ref{sec:est} derives uniform estimates and proves compactness for the transformed system.   
Section~\ref{prf} contains the convergence analysis, based on a nonlinear change of variables that enables the identification of the limiting dynamics.
Finally, subsection~\ref{subsec:marcus-drift} discusses the pathwise structure of the limiting equation and its consistency with the Marcus formulation of jump-driven stochastic integrals.


\section{Preliminaries}\label{sec:pre}

In this section, we introduce the functional analytic setting and the
standing assumptions used throughout the paper. We follow the framework of
the stochastic wave equation with homogeneous Dirichlet boundary condition,
and formulate the equation as an abstract evolution equation on the
corresponding phase space. We also recall the Meyer--Zheng topology and a
tightness criterion which will be used in the proof of the
Smoluchowski--Kramers approximation.

\subsection{Functional setting}

Let \(\mathcal O\subset \mathbb R^d\), \(d\ge1\), be a bounded domain with
boundary \(\partial\mathcal O\) of class \(C^3\). Throughout the paper, we
consider homogeneous Dirichlet boundary conditions on \(\partial\mathcal O\).

We denote by \( H:=L^2(\mathcal O)\) the Hilbert space endowed with the inner product
\[
\langle u,v\rangle_H := \int_{\mathcal O}u(x)v(x)\,dx, \qquad \|u\|_H:=\langle u,u\rangle_H^{1/2}.
\]
Following the usual notation for the Dirichlet Laplacian, we write \(H^1\)
for the completion of \(C_0^\infty(\mathcal O)\) with respect to the norm
\[
\|u\|_{H^1}^2:=\|\nabla u\|_H^2=\int_{\mathcal O}|\nabla u(x)|^2\,dx .
\]
Thus \(H^1=H_0^1(\mathcal O)\). We denote by \(H^{-1}\) the dual space of \(H^1\). Then \(H^1\subset H\subset H^{-1},\)
where the embeddings are continuous and compact. For \(f\in H^{-1}\) and \(\phi\in H^1\), we denote by
\(\langle f,\phi\rangle_{H^{-1},H^1}\) the duality between \(H^{-1}\) and \(H^1\). 

We shall use the phase spaces \(\mathcal H:=H\times H^{-1},~\mathcal H_1:=H^1\times H.\)
Let \(\{e_k\}_{k\ge1}\subset H^1\) be the complete orthonormal basis of
\(H\) consisting of eigenfunctions of the Dirichlet Laplacian. We denote by
\(\{-\alpha_k\}_{k\ge1}\) the corresponding eigenvalues,   \(\Delta e_k=-\alpha_k e_k,~ k\ge1,\)
where \(0<\alpha_1\le \alpha_2\le\cdots,~\alpha_k\to\infty .\)

For \(u\in H^{-1}\), its Fourier coefficients relative to the basis
\(\{e_k\}_{k\ge1}\) are defined by \(u_k:=\langle u,e_k\rangle_{H^{-1},H^1},~ k\ge1.\)
If \(u\in H\), then this definition coincides with the usual Fourier coefficient \(u_k=\langle u,e_k\rangle_H.\)
For \(u=\sum_{k=1}^\infty u_k e_k,\)  we have \(\|u\|_{H^1}^2=\sum_{k=1}^\infty \alpha_k |u_k|^2,~ u\in H^1,\)
\(\|u\|_H^2=\sum_{k=1}^\infty |u_k|^2,~ u\in H,\) and \(\|u\|_{H^{-1}}^2=\sum_{k=1}^\infty \alpha_k^{-1}|u_k|^2, ~ u\in H^{-1}.\)
In particular, the Poincar\'e inequalities \(\|u\|_H \le\frac1{\sqrt{\alpha_1}}\|u\|_{H^1},~ u\in H^1,\)
and \(\|u\|_{H^{-1}}\le\frac1{\sqrt{\alpha_1}}\|u\|_H,~ u\in H,\) hold.

For \(N\in\mathbb N\), define the finite-dimensional projection \(P_Nu:=\sum_{k=1}^N u_k e_k .\)
Then \(P_N\) is well defined on \(H^{-1}\), \(H\), and \(H^1\). Moreover,
for \(u\in H^{-1}\), \(\|(I-P_N)u\|_{H^{-1}}^2=\sum_{k=N+1}^{\infty}\alpha_k^{-1}|u_k|^2.\)

\subsection{Noises and coefficients}

Let \((\Omega,\mathcal F,\{\mathcal F_t\}_{t\ge0},\mathbb P)\) be a filtered probability space satisfying the usual conditions. Let
\(\{W^Q(t)\}_{t\ge0}\) be a \(Q\)-Wiener process on \(H\), adapted to \(\{\mathcal F_t\}_{t\ge0}\). The reproducing kernel Hilbert space
associated with \(Q\) is denoted by \(H^Q\). We write \(L_2(H^Q,H)\) for the space of Hilbert--Schmidt operators from \(H^Q\) to \(H\).

Let \(\mathcal N(dt,dz)\) be a Poisson random measure on
\(\mathbb R^m\) with intensity measure \(\nu(dz)dt\). Its compensated
Poisson random measure is denoted by \(\tilde{\mathcal N}(dt,dz):=\mathcal N(dt,dz)-\nu(dz)dt .\)
We assume that \(\mathcal N\) is adapted to the filtration \(\{\mathcal F_t\}_{t\ge0}\) and is independent of the Wiener noise.
Set \(Z_1:=\{z\in\mathbb R^m:\ |z|\le1\}.\)

We impose the following assumptions.

\begin{Assumption}\label{H1}
The measure \(\nu\) is a L\'evy measure on \(\mathbb R^m\), 
\[
\nu(\{0\})=0, \qquad \int_{\mathbb R^m}(1\wedge |z|^2)\,\nu(dz)<\infty .
\]
Moreover, \(\nu\) has finite second moment, that is \(\int_{\mathbb R^m}|z|^2\,\nu(dz)<\infty .\)
\end{Assumption}

\begin{Assumption}\label{H2}
The mappings \(\sigma:H\to L_2(H^Q,H),~F:H\to H,~G:H\times\mathbb R^m\to H\)
are measurable. Moreover, there exist constants \(L_\sigma,L_F,L_G>0\) such that, for all
\(u,u_1,u_2\in H\) and \(z\in\mathbb R^m\),
\begin{align*}
\|\sigma(u_1)-\sigma(u_2)\|_{L_2(H^Q,H)}&\le L_\sigma\|u_1-u_2\|_H,&\|\sigma(u)\|_{L_2(H^Q,H)}&\le L_\sigma(1+\|u\|_H),\\
\|F(u_1)-F(u_2)\|_H&\le L_F\|u_1-u_2\|_H,&\|F(u)\|_H&\le L_F(1+\|u\|_H),\\
\|G(u_1,z)-G(u_2,z)\|_H^2&\le L_G\|u_1-u_2\|_H^2|z|^2,&\|G(u,z)\|_H^2&\le L_G(1+\|u\|_H)^2|z|^2 .
\end{align*}
\end{Assumption}

\begin{Assumption}\label{H3}
The damping coefficient \(\gamma:\mathbb R\to\mathbb R\) belongs to
\(C_b^1(\mathbb R)\) and satisfies
\[
0<\gamma_0\le \gamma(r)\le \gamma_1,\qquad|\gamma'(r)|\le \gamma_1,\qquad r\in\mathbb R,
\]
for some constants \(\gamma_0,\gamma_1>0\). For \(u\in H\),   \(\gamma(u)(x):=\gamma(u(x)).\)
\end{Assumption}

\begin{remark}\label{rem:G-estimates}
By Assumptions~\ref{H1}--\ref{H2},  letting \(C_G:=L_G\int_{\mathbb R^m}|z|^2\,\nu(dz)<\infty,\)
then, for all \(u,u_1,u_2\in H\),
\[
\int_{\mathbb R^m}\|G(u_1,z)-G(u_2,z)\|_H^2\,\nu(dz) \le C_G\|u_1-u_2\|_H^2,
\]
\[
\int_{\mathbb R^m}\|G(u,z)\|_H^2\,\nu(dz) \le C_G(1+\|u\|_H)^2.
\]
Moreover, \(\nu(\mathbb R^m\setminus Z_1)\le\int_{\mathbb R^m\setminus Z_1}|z|^2\,\nu(dz)<\infty.\)
 By the Cauchy--Schwarz inequality,
\[
\int_{\mathbb R^m\setminus Z_1}\|G(u,z)\|_H\,\nu(dz) \le C(1+\|u\|_H).
\]
\end{remark}
Let \(Z^\epsilon(t):=(u^\epsilon(t),v^\epsilon(t)).\) Define the linear operator \(\tilde A^\epsilon\) on \(\mathcal H\) by
\[
\tilde A^\epsilon(u,v):=\left(v,\frac1\epsilon\Delta u\right),\qquad D(\tilde A^\epsilon)=\mathcal H_1.
\]
For \(Z^\epsilon=(u^\epsilon,v^\epsilon)\), define
\[
\tilde B(Z^\epsilon):=\left(0,\frac1\epsilon\bigl(F(u^\epsilon)-\gamma(u^\epsilon)v^\epsilon\bigr)\right),
\]
\[
\tilde \sigma(Z^\epsilon)
:=\left(0,\frac1\epsilon\sigma(u^\epsilon)\right),\qquad\tilde G(Z^\epsilon,z):=\left(0,\frac1\epsilon G(u^\epsilon,z)\right).
\]
Then equation~\eqref{e:SLE-11} can be written as the abstract stochastic
evolution equation
\begin{align}\label{e:X-eps-12}
dZ^\epsilon(t)
&=\tilde A^\epsilon Z^\epsilon(t)\,dt+\tilde B(Z^\epsilon(t))\,dt+\tilde\sigma(Z^\epsilon(t))\,dW^Q(t)+\int_{Z_1}\tilde G(Z^\epsilon(t-),z)\,\tilde{\mathcal N}(dt,dz)\nonumber\\
&\quad+\int_{\mathbb R^m\setminus Z_1}\tilde G(Z^\epsilon(t-),z)\,\mathcal N(dt,dz).
\end{align}
We denote by \(\{S^\epsilon(t)\}_{t\ge0}\) the \(C_0\)-semigroup generated by \(\tilde A^\epsilon\) on \(\mathcal H\).

\begin{definition}[Mild solution]\label{def:P2}
A process \(Z^\epsilon=(u^\epsilon,v^\epsilon) \in L^2\bigl(\Omega;D([0,T];\mathcal H_1)\bigr) \)
is called a mild solution to equation~\eqref{e:X-eps-12} if, for every \(t\in[0,T]\), it holds \(\mathbb P\)-almost surely that
\begin{align}\label{mild-1}
Z^\epsilon(t)
&=S^\epsilon(t)Z(0)+\int_0^t S^\epsilon(t-s)\tilde B(Z^\epsilon(s))\,ds+ \int_0^t S^\epsilon(t-s)\tilde\sigma(Z^\epsilon(s))\,dW^Q(s) \nonumber\\
&\quad+ \int_0^t\int_{Z_1} S^\epsilon(t-s)\tilde G(Z^\epsilon(s-),z)\, \tilde{\mathcal N}(ds,dz) \nonumber\\
&\quad+\int_0^t\int_{\mathbb R^m\setminus Z_1} S^\epsilon(t-s)\tilde G(Z^\epsilon(s-),z)\, \mathcal N(ds,dz),
\end{align}
where \(Z(0)=(u_0,v_0)\in\mathcal H_1.\)
\end{definition}

\begin{definition}[Weak solution to the limiting equation]\label{def:P3}
An adapted process \( u\in L^2\bigl(\Omega;D([0,T];H)\bigr) \cap L^2\bigl(\Omega;L^2(0,T;H^1)\bigr)\)
is called a weak solution to equation~\eqref{LAST} if, for every \(\psi\in C_0^\infty(\mathcal O)\) and every \(t\in[0,T]\), it holds
\(\mathbb P\)-almost surely that
\begin{align}\label{def-g-1}
\langle u(t),\psi\rangle_H &=\langle u_0,\psi\rangle_H-\int_0^t\int_{\mathcal O}\nabla u(s,\xi)\cdot\nabla\left(\frac{\psi(\xi)}{\gamma(u(s,\xi))}\right)d\xi ds\nonumber\\
&\quad+\int_0^t\left\langle\frac{F(u(s))}{\gamma(u(s))},\psi\right\rangle_Hds+\int_0^t\left\langle\frac{\sigma(u(s))}{\gamma(u(s))}\,dW^Q(s),\psi\right\rangle_H\nonumber\\
&\quad-\int_0^t\left\langle\frac{\gamma'(u(s))}{2\gamma(u(s))^3}\sigma(u(s))\sigma^*(u(s)),\psi\right\rangle_Hds\nonumber\\
&\quad+\int_0^t\int_{Z_1}\left\langle\frac{G(u(s-),z)}{\gamma(u(s-))},\psi\right\rangle_H\tilde{\mathcal N}(ds,dz)\nonumber\\
&\quad+\int_0^t\int_{\mathbb R^m\setminus Z_1}\left\langle\frac{G(u(s-),z)}{\gamma(u(s-))},\psi\right\rangle_H\mathcal N(ds,dz)\nonumber\\
&\quad+\left\langle\sum_{0<s\le t}\left[u(s)-u(s-)-\frac1{\gamma(u(s-))}\int_{u(s-)}^{u(s)}\gamma(r)\,dr\right],\psi \right\rangle_H .
\end{align}
\end{definition}

\begin{remark}
The expression \(\int_{u(s-)}^{u(s)}\gamma(r)\,dr\) in~\eqref{def-g-1} is understood pointwise in the spatial variable.  
It denotes the function \(x\mapsto \int_{u(s-,x)}^{u(s,x)}\gamma(r)\,dr .\) Moreover, the term \(\sigma(u)\sigma^*(u)\)
in~\eqref{LAST} and~\eqref{def-g-1} should be understood according to the
specific structure of the Wiener noise coefficient. In the case of multiplicative noise, it is the corresponding pointwise quadratic variation
density. For a general Hilbert--Schmidt coefficient, this term should be interpreted as the trace correction associated with the It\^o formula.
\end{remark}

\subsection{The Meyer--Zheng topology}

We next recall the Meyer--Zheng topology~\cite{BlathHammerOrtgiese2016,MeyerZheng1984}, also called the pseudo-path
topology. Let \(E\) be a separable Hilbert space. We denote by
\(D([0,T];E)\) the space of all \(E\)-valued c\`adl\`ag functions on
\([0,T]\). For \(x,y\in D([0,T];E)\), define
\[
d_{\rm MZ}(x,y):=\int_0^T\left(1\wedge\|x(t)-y(t)\|_E\right)\,dt .
\]
The topology induced by \(d_{\rm MZ}\) is the topology of convergence in
Lebesgue measure on \([0,T]\) with values in \(E\). That is
\[
x_n\to x \quad\text{in the Meyer--Zheng topology}
\]
if and only if
\[
x_n\to x \quad\text{in Lebesgue measure on }[0,T]\text{ with values in }E.
\]
Equivalently, to each \(x\in D([0,T];E)\) one associates the measure
\(\mu_x(dt,d\xi):=\frac1T\,dt\,\delta_{x(t)}(d\xi)\) on \([0,T]\times E\). The Meyer--Zheng topology is the topology generated
by the weak convergence of the measures \(\{\mu_x\}\). In particular, this topology is weaker than the usual Skorokhod \(J_1\) topology.
In this paper, we shall mainly use the Meyer--Zheng topology on \(D([0,T];H^{-1}).\)

\begin{definition}
Let \(X=\{X(t)\}_{t\in[0,T]}\) be a real-valued adapted c\`adl\`ag process.
Its conditional variation on \([0,T]\) is defined by
\[
{\rm CV}_T(X):=\sup_{\pi}\mathbb E\left[\sum_{i=0}^{n-1}\left|\mathbb E\left[X(t_{i+1})-X(t_i)\mid\mathcal F_{t_i}\right]\right|+|X(T)|\right],
\]
where the supremum is taken over all partitions \(\pi:~ 0=t_0<t_1<\cdots<t_n=T.\) If \({\rm CV}_T(X)<\infty,\) then \(X\) is called a quasimartingale.
\end{definition}

\begin{theorem}[Meyer--Zheng tightness criterion]\label{thm:MZ-tightness}
Let \(\{X^n\}_{n\ge1}\) be a sequence of real-valued adapted c\`adl\`ag
processes on \([0,T]\). Suppose that each \(X^n\) is a quasimartingale and
\[
\sup_{n\ge1}{\rm CV}_T(X^n)<\infty.
\]
Then the laws of \(\{X^n\}_{n\ge1}\) are tight in
\(D([0,T];\mathbb R)\) endowed with the Meyer--Zheng topology.
\end{theorem}

The preceding criterion is real-valued. For infinite-dimensional
processes, we shall apply it to finite-dimensional projections and then
control the  tail. The following form will be useful later.

\begin{lemma}[Projection criterion for Meyer--Zheng tightness]
\label{lem:MZ-projection-tightness}
Let \(\{Y^n\}_{n\ge1}\) be a sequence of \(H^{-1}\)-valued adapted c\`adl\`ag processes. Assume that, for every \(k\ge1\),
\[
\sup_{n\ge1}{\rm CV}_T\bigl(\langle Y^n,e_k\rangle_{H^{-1},H^1}\bigr)<\infty,
\]
and that, for every \(\eta>0\),
\[
\lim_{N\to\infty}\sup_{n\ge1}\mathbb P\left(\int_0^T\|(I-P_N)Y^n(t)\|_{H^{-1}}\,dt>\eta\right)=0.
\]
Then the laws of \(\{Y^n\}_{n\ge1}\) are tight in \(D([0,T];H^{-1})\) endowed with the Meyer--Zheng topology.
\end{lemma}

\begin{proof}
For each fixed \(N\in\mathbb N\), the process \(P_NY^n\) is
finite-dimensional. By Theorem~\ref{thm:MZ-tightness}, applied to each
coordinate
\[
\langle Y^n,e_k\rangle_{H^{-1},H^1},\qquad 1\le k\le N,
\]
the laws of \(\{P_NY^n\}_{n\ge1}\) are tight in \(D([0,T];P_NH^{-1})\) endowed with the Meyer--Zheng topology.
Moreover,
\[
d_{\rm MZ}(Y^n,P_NY^n)\le\int_0^T\|(I-P_N)Y^n(t)\|_{H^{-1}}\,dt .
\]
The assumed tail estimate therefore implies that \(Y^n-P_NY^n\) is negligible in the Meyer--Zheng topology, uniformly in \(n\), as
\(N\to\infty\). Combining the finite-dimensional tightness with this tail estimate yields the desired tightness of \(\{Y^n\}_{n\ge1}\) in
\(D([0,T];H^{-1})\).
\end{proof}

\textbf{Notation}

Throughout the paper, \(C\) denotes a generic positive constant whose value
may change from line to line. Constants with subscripts, such as
\(C_T\), \(C_{\epsilon,T}\), \(C_\gamma\), may depend only on the indicated
parameters.

All differential operators \(\nabla\), \({\rm div}\), and \(\Delta\) aretaken with respect to the spatial variable \(x\). We denote by
\(\mathcal L_\xi\) the law of a random variable or stochastic process \(\xi\).

We now state the following well-posedness result for the stochastic system~\eqref{e:SLE-11}, formulated in terms of mild solutions as in Definition~\ref{def:P2}. The proof relies on a standard localization and truncation argument combined with fixed-point methods.

\begin{proposition}\label{well-pose-1}
Let $T>0$, and assume that Assumptions~\ref{H1}--\ref{H3} hold. Then for every initial condition $Z(0) = (u_0, v_0) \in \mathcal H$ and every $\epsilon \in (0,1)$, there exists a unique mild solution \(Z^{\epsilon} \in L^2\big(\Omega; \mathcal D([0,T]; \mathcal H)\big)\)
to equation~\eqref{e:X-eps-12}. Moreover, this mild solution coincides with the unique weak solution to~\eqref{e:X-eps-12}.
\end{proposition}
\begin{proof}
Omitting the superscript $\epsilon$, and under Assumptions~\ref{H1}--\ref{H3},
the nonlinear term $\tilde B(Z)$ is locally Lipschitz on $\mathcal H$ due to
the presence of the damping term $\gamma(u)v$.
More precisely, for every $R>0$ there exists a constant $L_{K,R}>0$ such that
\begin{equation}\label{assump-1}
\|\tilde B(Z^1)-\tilde B(Z^2)\|_{\mathcal H}
\le L_{K,R}\|Z^1-Z^2\|_{\mathcal H},
\end{equation}
for all $Z^1,Z^2\in\mathcal H$ satisfying
$\|Z^1\|_{\mathcal H}+\|Z^2\|_{\mathcal H}\le R$.

For $n\ge1$, we define the stopping time \(\tau_Z^n:=\inf\big\{t>0:\ \|Z(t)\|_{\mathcal H}\ge n\big\},\)
where $\|Z\|_{\mathcal H}:=\|u\|_{H^1}+\|v\|_H$. By equation~\eqref{e:X-eps-12}, we obtain
\begin{eqnarray}\label{well-2}
\mathbb E\sup_{t\in[0,T\wedge\tau_Z^n]}\|Z(t)\|_{\mathcal H}^2
&\le& 4\Bigg[\mathbb E\|S(t)Z(0)\|_{\mathcal H}^2+ \sum_{i=1}^4 I_i\Bigg],
\end{eqnarray}
where
\begin{align*}
I_1 &:= \mathbb E\sup_{t\in[0,T\wedge\tau_Z^n]}
\Big\|\int_0^t S(t-s)\tilde B(Z(s))\,ds\Big\|_{\mathcal H}^2,\\
I_2 &:= \mathbb E\sup_{t\in[0,T\wedge\tau_Z^n]}
\Big\|\int_0^t S(t-s)\tilde\sigma(Z(s))\,dW^Q(s)\Big\|_{\mathcal H}^2,\\
I_3 &:= \mathbb E\sup_{t\in[0,T\wedge\tau_Z^n]}
\Big\|\int_0^t\!\!\int_{Z_1}
S(t-s)\tilde G(Z(s-),z)\,\tilde{\mathcal N}(ds,dz)\Big\|_{\mathcal H}^2,\\
I_4 &:= \mathbb E\sup_{t\in[0,T\wedge\tau_Z^n]}
\Big\|\int_0^t\!\!\int_{\mathbb R^m\setminus Z_1}
S(t-s)\tilde G(Z(s-),z)\,\mathcal N(ds,dz)\Big\|_{\mathcal H}^2 .
\end{align*}
We estimate the terms $I_i$ separately.
For $I_1$, by Assumptions~\ref{H2}--\ref{H3} and the boundedness of the semigroup $\{S(t)\}_{t\ge0}$, we have
\begin{eqnarray}\label{well-3}
I_1 &\le&C_S^2\,\mathbb E\sup_{t\in[0,T\wedge\tau_Z^n]}\int_0^t\|\tilde B(Z(s))\|_{\mathcal H}^2\,ds \nonumber\\
&\le&\frac{4C_S^2}{\epsilon^2}\mathbb E\sup_{t\in[0,T\wedge\tau_Z^n]}\int_0^t\big(\|F(u(s))\|_H^2+\gamma_1^2\|v(s)\|_H^2\big)\,ds \nonumber\\
&\le& C_{\gamma_1,L_F}\mathbb E\int_0^T\sup_{r\in[0,s\wedge\tau_Z^n]}\|Z(r)\|_{\mathcal H}^2\,ds .
\end{eqnarray}
For $I_2$, by the Burkholder--Davis--Gundy inequality and Assumption~\ref{H2}, we obtain
\begin{eqnarray}\label{well-4}
I_2 \le C_S^2 C_2 L_\sigma^2\Bigg(1+\int_0^T\mathbb E\sup_{s\in[0,t\wedge\tau_Z^n]}\|Z(s)\|_{\mathcal H}^2\,dt\Bigg).
\end{eqnarray}
For the jump terms $I_3$ and $I_4$, by Assumptions~\ref{H1}--\ref{H2} and Kunita’s first inequality (see, e.g.,~\cite{Applebaum}), we obtain
\begin{eqnarray}\label{well-levy-1}
I_3+I_4&\le& 2C_S^2 C_2\int_0^T\mathbb E\sup_{s\in[0,t\wedge\tau_Z^n]}\int_{\mathbb R^m}\|\tilde G(Z(s-),z)\|_{\mathcal H}^2\,\nu(dz)\,dt\nonumber\\
&\le& 2C_S^2 C_2 L_G\int_0^T\mathbb E\sup_{s\in[0,t\wedge\tau_Z^n]}\|Z(s)\|_{\mathcal H}^2\,dt .
\end{eqnarray}
Collecting estimates~\eqref{well-2}--\eqref{well-4} and \eqref{well-levy-1}, we arrive at
\begin{eqnarray}\label{well-gronwall}
\mathbb E\sup_{t\in[0,T\wedge\tau_Z^n]}\|Z(t)\|_{\mathcal H}^2
&\le&C_S\|Z(0)\|_{\mathcal H}^2+ C_{S,\gamma_1,\sigma,F,G}\int_0^T\mathbb E\sup_{s\in[0,t\wedge\tau_Z^n]}\|Z(s)\|_{\mathcal H}^2\,dt .
\end{eqnarray}
By Gronwall’s inequality, it follows that
\begin{equation}\label{well-02-}
\mathbb E\sup_{t\in[0,T\wedge\tau_Z^n]}\|Z(t)\|_{\mathcal H}^2\le \bar C,\qquad n\ge1,
\end{equation}
where
\(\bar C:= C_S \exp\!\big(C_{S,\gamma_1,\sigma,F,G}T\big)\,\mathbb E\|Z(0)\|_{\mathcal H}^2 .\) Consequently,
\[
n^2\,\mathbb P(\tau_Z^n\le T)\le \mathbb E\|Z(T\wedge\tau_Z^n)\|_{\mathcal H}^2 \le \bar C,
\]
which yields
\[
\mathbb P(\tau_Z^n\le T)\le \frac{\bar C}{n^2} \longrightarrow 0, \qquad n\to\infty .
\]
Here and throughout the paper, we adopt the convention $\inf\emptyset=\infty$.
Moreover, since the mapping $Z\mapsto\tau_Z^n$ is continuous on $\mathcal D([0,T];\mathcal H)$,
$\tau_Z^n$ is a stopping time whenever $Z$ is adapted.

We introduce a localization procedure in order to obtain globally Lipschitz
coefficients.
For $n\ge1$ and $Z\in\mathcal D([0,T];\mathcal H)$, define \(\varphi_n(Z)(t):=Z(t\wedge\tau_Z^n),~ t\in[0,T],\)
where $\tau_Z^n:=\inf\{t>0:\|Z(t)\|_{\mathcal H}\ge n\}$. Then the mapping $\varphi_n:\mathcal D([0,T];\mathcal H)\to
\mathcal D([0,T];\mathcal H)$ is continuous and hence measurable.

For $\xi\in\mathcal D([0,T];\mathcal H)$, we define the truncated coefficients
\[
\tilde A^n\xi:=\tilde A\big(\varphi_n(\xi)\big),~~
\tilde B^n(\xi):=\tilde B\big(\varphi_n(\xi)\big),~~
\tilde\sigma^n(\xi):=\tilde\sigma\big(\varphi_n(\xi)\big),~~
\tilde G^n(\xi,z):=\tilde G\big(\varphi_n(\xi),z\big).
\]
By construction, $\tilde B^n$, $\tilde\sigma^n$ and $\tilde G^n$ are globally Lipschitz on $\mathcal H$. We consider the truncated equation
\begin{eqnarray}\label{well-01}
dZ^n(t)&=& \tilde A^n Z^n(t)\,dt+ \tilde B^n\big(Z^n(t)\big)\,dt+ \tilde\sigma^n\big(Z^n(t)\big)\,dW^Q(t) \nonumber\\
&&+ \int_{Z_1}\tilde G^n\big(Z^n(t-),z\big)\,\tilde{\mathcal N}(dt,dz)+ \int_{\mathbb R^m\setminus Z_1}\tilde G^n\big(Z^n(t-),z\big)\,{\mathcal N}(dt,dz),
\end{eqnarray}
with initial condition $Z^n(0)=Z(0)$. By standard results on stochastic evolution equations with L\'evy noise
(see~\cite{linjunbo}), equation~\eqref{well-01} admits a unique weak solution $Z^n$.
Moreover, by~\cite[Theorem~8.2]{peszat}, this solution is also the unique mild solution, which satisfies
\begin{eqnarray}\label{well-1}
Z^n(t)&=& S(t)Z(0)+ \int_0^t S(t-s)\tilde B\big(Z^n(s\wedge\tau_Z^n)\big)\,ds+ \int_0^t S(t-s)\tilde\sigma\big(Z^n(s\wedge\tau_Z^n)\big)\,dW^Q(s)\nonumber\\
&&+ \int_0^t\!\!\int_{Z_1}S(t-s)\tilde G\big(Z^n((s\wedge\tau_Z^n)-),z\big)\,\tilde{\mathcal N}(ds,dz)\nonumber\\
&&+ \int_0^t\!\!\int_{\mathbb R^m\setminus Z_1}S(t-s)\tilde G\big(Z^n((s\wedge\tau_Z^n)-),z\big)\,{\mathcal N}(ds,dz).
\end{eqnarray}
As in the proof of Proposition~\ref{well-pose-1}, we obtain the uniform estimate
\begin{equation}\label{well-02}
\mathbb E\sup_{t\in[0,T]}\|Z^n(t)\|_{\mathcal H}^2 \le \bar C, \qquad n\ge1,
\end{equation}
where \(\bar C:= C_S \exp\!\big(C_{S,\gamma_1,\sigma,F,G}T\big)\,\mathbb E\|Z(0)\|_{\mathcal H}^2 .\)
In particular, $\{Z^n\}_{n\ge1}$ is bounded in
$L^2\big(\Omega;\mathcal D([0,T];\mathcal H)\big)$ uniformly in $n$.
Moreover,
\[
n^2\,\mathbb P(\tau_Z^n\le T) \le\mathbb E\|Z^n(T\wedge\tau_Z^n)\|_{\mathcal H}^2 \le \bar C,
\]
which implies
\[
\mathbb P(\tau_Z^n\le T) \le \frac{\bar C}{n^2}\longrightarrow 0,\qquad n\to\infty .
\]
Consequently, $\tau_Z^n\to\infty$ almost surely as $n\to\infty$, and the solutions $Z^n$ converge to a global solution $Z$ of
equation~\eqref{e:X-eps-12}. This limit $Z$ is the unique mild solution to~\eqref{e:X-eps-12}, and by \cite[Theorem~8.2]{peszat}, it is also the unique weak solution.

\end{proof}

\section{Uniform estimates and compactness}\label{sec:est}

\subsection{Energy estimates and moment bounds}
In this subsection, we derive  the uniform energy estimates for the solution $(u^\epsilon,v^\epsilon)$ of equation~\eqref{e:SLE-11}, which play a  role in the
analysis of the small--mass limit $\epsilon\to0$. The estimates obtained here are uniform with respect to $\epsilon$ and provide
sufficient compactness to identify the limiting dynamics.

\begin{lemma}\label{lemma1-33}
Assume that Assumptions~\ref{H1}--\ref{H3} hold.
Let $(u^{\epsilon},v^{\epsilon})$ be the unique solution to
equation~\eqref{e:SLE-11} with initial condition
$(u_0,v_0)\in\mathcal H$.
Then, for every $T>0$ and all sufficiently small $\epsilon>0$, there exists
a constant $C_T>0$, depending only on
$T$, $L_\sigma$, $L_F$, $L_G$, $u_0$, $v_0$, $\gamma_0$, and $\gamma_1$, such that
\begin{equation}\label{a-priori-u}
\mathbb E\sup_{t\in[0,T]}\|u^{\epsilon}(t)\|_H^2 \le C_T,
\end{equation}
and
\begin{equation}\label{a-priori-v}
\sup_{t\in[0,T]}
\Big(
\epsilon\,\mathbb E\|v^{\epsilon}(t)\|_H^2
+ \mathbb E\|\nabla u^{\epsilon}(t)\|_H^2
\Big)
\le C_T .
\end{equation}
\end{lemma}

\begin{proof}
Notice that, since $du^{\epsilon}(t)=v^{\epsilon}(t)\,dt$, we have
\begin{equation}\label{lemma1-1}
d\|u^{\epsilon}(t)\|_{H}^2
=2\langle u^{\epsilon}(t),du^{\epsilon}(t)\rangle_{H}
=2\langle u^{\epsilon}(t),v^{\epsilon}(t)\rangle_{H}\,dt .
\end{equation}
Moreover, applying It\^{o}'s formula for semimartingales with jumps to the
process $\langle u^{\epsilon}(t),v^{\epsilon}(t)\rangle_{H}$ and using again
$du^{\epsilon}(t)=v^{\epsilon}(t)\,dt$, we obtain
\begin{align}
\epsilon\, d\langle u^{\epsilon}(t),v^{\epsilon}(t)\rangle_{H}
&=\epsilon\langle du^{\epsilon}(t),v^{\epsilon}(t)\rangle_{H}
+\epsilon\langle u^{\epsilon}(t-),dv^{\epsilon}(t)\rangle_{H} \nonumber\\
&=\epsilon\|v^{\epsilon}(t)\|_{H}^2\,dt
+\Big\langle u^{\epsilon}(t),\Delta u^{\epsilon}(t)\Big\rangle_H\,dt
-\Big\langle u^{\epsilon}(t),\gamma(u^{\epsilon}(t))v^{\epsilon}(t)\Big\rangle_H\,dt \nonumber\\
&\quad
-\Big\langle u^{\epsilon}(t),F(u^{\epsilon}(t))\Big\rangle_H\,dt
+\Big\langle u^{\epsilon}(t),\sigma(u^{\epsilon}(t))\,dW^Q(t)\Big\rangle_H \nonumber\\
&\quad
+\int_{Z_1}\Big\langle u^{\epsilon}(t-),G(u^{\epsilon}(t-),z)\Big\rangle_H\,\tilde{\mathcal N}(dt,dz)\nonumber\\
&\quad+\int_{\mathbb R^m\setminus Z_1}\Big\langle u^{\epsilon}(t-),G(u^{\epsilon}(t-),z)\Big\rangle_H\, {\mathcal N}(dt,dz).
\label{lemma1-cross}
\end{align}
Here we used that the quadratic covariation term vanishes since $u^{\epsilon}$
has finite variation. Using the identity
$\langle u^{\epsilon}(t),\Delta u^{\epsilon}(t)\rangle_H=-\|\nabla u^{\epsilon}(t)\|_H^2$,
we can rearrange~\eqref{lemma1-cross} to obtain
\begin{align}\label{lemma1-2-plusF}
\big\langle u^{\epsilon}(t),\gamma(u^{\epsilon}(t))v^{\epsilon}(t)\big\rangle_H\,dt
&= -\epsilon\, d\langle u^{\epsilon}(t),v^{\epsilon}(t)\rangle_H
-\|\nabla u^{\epsilon}(t)\|_H^2\,dt
+\big\langle u^{\epsilon}(t),F(u^{\epsilon}(t))\big\rangle_H\,dt  \nonumber\\
&\quad
+\epsilon\|v^{\epsilon}(t)\|_H^2\,dt
+\big\langle u^{\epsilon}(t),\sigma(u^{\epsilon}(t))\,dW^Q(t)\big\rangle_H \nonumber\\
&\quad
+\int_{Z_1}\big\langle u^{\epsilon}(t-),G(u^{\epsilon}(t-),z)\big\rangle_H\,\tilde{\mathcal N}(dt,dz)\nonumber\\
&\quad+\int_{\mathbb R^m\setminus Z_1}\big\langle u^{\epsilon}(t-),G(u^{\epsilon}(t-),z)\big\rangle_H\, {\mathcal N}(dt,dz).
\end{align}
Define
\[
\Phi(r):=\int_0^r x\,\gamma(x)\,dx,\qquad r\in\mathbb R,
\qquad\text{and}\qquad
\Gamma(u):=\int_{\mathcal O}\Phi\big(u(x)\big)\,dx,\qquad u\in H.
\]
By Assumption~\ref{H3}, for all $r\in\mathbb R$,
\(
0\le \frac{\gamma_0}{2}r^2\le \Phi(r)\le \frac{\gamma_1}{2}r^2,
\)
and hence, for all $u\in H$,
\(
0\le \frac{\gamma_0}{2}\|u\|_H^2 \le \Gamma(u)\le \frac{\gamma_1}{2}\|u\|_H^2 .
\)
Since $u^{\epsilon}$ has continuous paths and
$du^{\epsilon}(t)=v^{\epsilon}(t)\,dt$, the mapping
$t\mapsto \Gamma(u^{\epsilon}(t))$ is absolutely continuous and
\begin{eqnarray}
\frac{d}{dt}\Gamma(u^{\epsilon}(t))
=\int_{\mathcal O}\Phi'\big(u^{\epsilon}(t,x)\big)\,v^{\epsilon}(t,x)\,dx
&=&\int_{\mathcal O}u^{\epsilon}(t,x)\gamma\big(u^{\epsilon}(t,x)\big)v^{\epsilon}(t,x)\,dx \nonumber\\
&=&\langle u^{\epsilon}(t),\gamma(u^{\epsilon}(t))v^{\epsilon}(t)\rangle_H .
\end{eqnarray}
Therefore,
\begin{eqnarray}\label{lemma1-6}
\int_0^t \langle u^{\epsilon}(s),\gamma(u^{\epsilon}(s))v^{\epsilon}(s)\rangle_H\,ds
= \Gamma(u^{\epsilon}(t))-\Gamma(u_0)
\ge \frac{\gamma_0}{2}\|u^{\epsilon}(t)\|_H^2-\frac{\gamma_1}{2}\|u_0\|_H^2 .
\end{eqnarray}
Combining \eqref{lemma1-1}, \eqref{lemma1-2-plusF} and \eqref{lemma1-6}, we obtain
\begin{eqnarray}\label{lemma1-energy-ineq}
 \frac{\gamma_0}{2}\|u^{\epsilon}(t)\|_H^2
+\int_0^t \|\nabla u^{\epsilon}(s)\|_H^2\,ds 
&\le&
\frac{\gamma_1}{2}\|u_0\|_H^2
-\epsilon\langle u^{\epsilon}(t),v^{\epsilon}(t)\rangle_H
+\epsilon\langle u_0,v_0\rangle_H \nonumber\\
&&
+\epsilon\int_0^t \|v^{\epsilon}(s)\|_H^2\,ds
+\int_0^t \langle u^{\epsilon}(s),F(u^{\epsilon}(s))\rangle_H\,ds \nonumber\\
&&
+\int_0^t \langle u^{\epsilon}(s),\sigma(u^{\epsilon}(s))\,dW^Q(s)\rangle_H \nonumber\\
&&
+\int_0^t\int_{Z_1}\langle u^{\epsilon}(s-),G(u^{\epsilon}(s-),z)\rangle_H\,\tilde{\mathcal N}(ds,dz) \nonumber\\
&&
+\int_0^t\int_{\mathbb R^m\setminus Z_1}\langle u^{\epsilon}(s-),G(u^{\epsilon}(s-),z)\rangle_H\, {\mathcal N}(ds,dz).\nonumber\\
\end{eqnarray}
 Particularly, for $\epsilon\in(0,1)$ and some constant $C_0=\|(u_0,v_0)\|_{\mathcal H}^2$, by Young's inequality and
Assumption~\ref{H2}, we obtain
\begin{eqnarray}
&& \frac{\gamma_0}{4}\,\|u^{\epsilon}(t)\|_{H}^2
+ \int_0^t \|\nabla u^{\epsilon}(s)\|_{H}^2\,ds \nonumber\\
&\leq&
\gamma_1 C_0
+ L_F\Big(T + \int_0^t \|u^{\epsilon}(s)\|_{H}^2\,ds\Big)
+ \frac{4\epsilon^2}{\gamma_0}\,\|v^{\epsilon}(t)\|_{H}^2
+ \int_0^t \epsilon\,\|v^{\epsilon}(s)\|_{H}^2\,ds \nonumber\\
&&
+ \int_0^t \langle u^{\epsilon}(s),
\sigma(u^{\epsilon}(s))\,dW^Q(s)\rangle_H  + \int_0^t \Big\langle u^{\epsilon}(s),
\int_{Z_1} G(u^{\epsilon}(s-),z)\,\tilde{\mathcal N}(ds,dz) \Big\rangle_H \nonumber\\
&&
+ \int_0^t \Big\langle u^{\epsilon}(s),
\int_{\mathbb R^m\setminus Z_1} G(u^{\epsilon}(s-),z)\,{\mathcal N}(ds,dz) \Big\rangle_H .
\end{eqnarray}
Taking the supremum over $t\in[0,T]$ and expectations, and applying
the Burkholder--Davis--Gundy inequality to the Gaussian term together with Kunita's first inequality for the jump terms (see~\cite[Theorem~4.4.23]{Applebaum}),
we infer from Assumption~\ref{H2} that
\begin{eqnarray}
&& \mathbb{E}\sup_{t\in[0,T]}\|u^{\epsilon}(t)\|_{H}^2
+ \mathbb{E}\int_0^T \|\nabla u^{\epsilon}(s)\|_{H}^2\,ds \nonumber\\
&\leq&
C_T\Big[1 + \epsilon^2\,\mathbb{E}\sup_{t\in[0,T]}\|v^{\epsilon}(t)\|_{H}^2+ \epsilon\,\mathbb{E}\int_0^T \|v^{\epsilon}(s)\|_{H}^2\,ds\Big] \nonumber\\
&&+ C_{L_F,L_\sigma}\int_0^T \mathbb{E}\|u^{\epsilon}(s)\|_{H}^2\,ds+ 2\int_0^T \int_{\mathbb R^m} \|G(u^{\epsilon}(s-),z)\|_{H}^2\,\nu(dz)\,ds \nonumber\\
&\leq& C_T\Big[ 1 + \epsilon^2\,\mathbb{E}\sup_{t\in[0,T]}\|v^{\epsilon}(t)\|_{H}^2+ \epsilon\,\mathbb{E}\int_0^T \|v^{\epsilon}(s)\|_{H}^2\,ds
\Big]  + C\int_0^T \mathbb{E}\|u^{\epsilon}(s)\|_{H}^2\,ds .
\end{eqnarray}
Applying Gronwall's inequality yields
\begin{eqnarray}\label{lemma1-3}
&&  \mathbb{E}\sup_{t\in [0,T]}\|u^{\epsilon }(t)\|_{H}^2+  \mathbb{E}\sup_{t\in [0,T]}\int_0^t \|\nabla u^{\epsilon }(s)\|_{H}^2\,ds \nonumber\\
&\leq& C_{T,\sigma,G,F}\left(1+ \epsilon^2\,\mathbb{E}\sup_{t\in [0,T]}\| v^{\epsilon}(t)\|_{H}^2+ \epsilon\,\mathbb{E}\int_0^T \|v^{\epsilon}(s)\|_{H}^2\,ds\right).
\end{eqnarray}
 Next, we show that
\[
\epsilon^2 \mathbb{E}\sup_{t\in[0,T]}\|v^{\epsilon}(t)\|_{H}^2 + \epsilon \mathbb{E}\int_0^T \|v^{\epsilon}(s)\|_{H}^2\,ds \le C_T .
\]
We apply It\^o's formula to $\|v^{\epsilon}(t)\|_{H}^2$.
Recalling that
\[
\epsilon\, dv^{\epsilon}= \Delta u^{\epsilon}\,dt- \gamma(u^{\epsilon}) v^{\epsilon}\,dt+ F(u^{\epsilon})\,dt+ \sigma(u^{\epsilon})\,dW^Q
+ \int_{\mathbb R^m} G(u^{\epsilon},z)\,\tilde{\mathcal N}(dt,dz),
\]
we obtain
\begin{eqnarray}\label{e:v-v}
\epsilon\, d\|v^{\epsilon}(t)\|_{H}^2
&=& 2\Big(
\langle v^{\epsilon}(t),\Delta u^{\epsilon}(t)\rangle_H
- \langle v^{\epsilon}(t),\gamma(u^{\epsilon}(t))v^{\epsilon}(t)\rangle_H
+ \langle v^{\epsilon}(t),F(u^{\epsilon}(t))\rangle_H
\Big)\,dt \nonumber\\
&& + 2\langle v^{\epsilon}(t),\sigma(u^{\epsilon}(t))\,dW^Q(t)\rangle_H
+ \frac{1}{\epsilon}\|\sigma(u^{\epsilon}(t))\|_{L_2(H^Q,H)}^2\,dt \nonumber\\
&& + \epsilon\int_{\mathbb R^m}
\Big(
\|v^{\epsilon}(t-)+\tfrac{G(u^{\epsilon}(t-),z)}{\epsilon}\|_{H}^2
- \|v^{\epsilon}(t-)\|_{H}^2
\Big)\,\tilde{\mathcal N}(dt,dz) \nonumber\\
&& + \epsilon\int_{\mathbb R^m\setminus Z_1}
\Big(
\|v^{\epsilon}(t-)+\tfrac{G(u^{\epsilon}(t-),z)}{\epsilon}\|_{H}^2
- \|v^{\epsilon}(t-)\|_{H}^2
\Big)\,\nu(dz)\,dt \nonumber\\
&& + \epsilon\int_{Z_1}
\Big(
\|v^{\epsilon}(t-)+\tfrac{G(u^{\epsilon}(t-),z)}{\epsilon}\|_{H}^2
- \|v^{\epsilon}(t-)\|_{H}^2
\nonumber\\
&&- 2\big\langle \tfrac{G(u^{\epsilon}(t-),z)}{\epsilon},v^{\epsilon}(t)\big\rangle_H
\Big)\,\nu(dz)\,dt .
\end{eqnarray}

Integrating \eqref{e:v-v} over $[0,t]$ and using Assumptions~\ref{H2}--\ref{H3},
we obtain
\begin{eqnarray}\label{ito-es-1}
&& \epsilon\|v^{\epsilon}(t)\|_{H}^2
+ 2\gamma_0\int_0^t \|v^{\epsilon}(s)\|_{H}^2\,ds
+ \|\nabla u^{\epsilon}(t)\|_{H}^2 \nonumber\\
&\le&
\epsilon\|v_0\|_{H}^2 + \|\nabla u_0\|_{H}^2
+ 2\int_0^t \langle v^{\epsilon}(s),F(u^{\epsilon}(s))\rangle_H\,ds \nonumber\\
&& + \frac{2L_\sigma^2}{\epsilon}\int_0^t (1+\|u^{\epsilon}(s)\|_{H}^2)\,ds
+ 2\int_0^t \langle v^{\epsilon}(s),\sigma(u^{\epsilon}(s))\,dW^Q(s)\rangle_H \nonumber\\
&& + I_1 + I_2 + I_3 .
\end{eqnarray}
Applying the Burkholder--Davis--Gundy inequality and Assumption~\ref{H2},
there exists a positive constant $C_{\gamma_0}$ such that
\begin{eqnarray}\label{ito-l-1}
\mathbb{E}\sup_{t\in [0,T]}
\left|\int_0^t \langle v^\epsilon(s),\sigma(u^\epsilon(s))\,dW^Q(s)\rangle_H\right|
\leq
\frac{\gamma_0}{12}\int_0^T \mathbb{E}\|v^\epsilon(s)\|_H^2\,ds
+ C_{\gamma_0}\int_0^T \mathbb{E}\|u^\epsilon(s)\|_H^2\,ds .
\end{eqnarray}
Moreover, by Kunita's first inequality
\cite[Theorem~4.4.23]{Applebaum} and Assumption~\ref{H2}, we have
\begin{eqnarray}\label{ito-l-2}
\mathbb{E}\sup_{t\in [0,T]} I_2
&\leq&
\frac{\gamma_0}{6}\int_0^T \mathbb{E}\|v^\epsilon(s)\|_H^2\,ds
+ C_{\gamma_0}\int_0^T \mathbb{E}\|u^\epsilon(s)\|_H^2\,ds
\nonumber\\
&&\qquad
+ \frac{L_G^2}{\epsilon}
\Big(1+\int_0^T \mathbb{E}\|u^\epsilon(s)\|_H^2\,ds\Big),
\end{eqnarray}
and
\begin{eqnarray}\label{ito-l-3}
\mathbb{E}\sup_{t\in [0,T]} I_3
\le
\frac{L_G^2}{\epsilon}
\Big(1+\int_0^T \mathbb{E}\|u^\epsilon(s)\|_H^2\,ds\Big).
\end{eqnarray}
Similarly, applying Assumption~\ref{H2} and Kunita's first inequality,
there exists a positive constant $C_{\gamma_0}$ such that
\begin{eqnarray}\label{ito-l-4}
&&\mathbb{E}\sup_{t\in [0,T]} I_1
\le
\int_0^T\int_{\mathbb{R}^m}
\Big(
\frac{\|G(u^\epsilon(s-),z)\|_H^2}{\epsilon}
+ 2\big|\langle G(u^\epsilon(s-),z),v^\epsilon(s)\rangle_H\big|
\Big)\nu(dz)\,ds
\nonumber\\
&\leq&
\frac{\gamma_0}{6}\int_0^T \mathbb{E}\|v^\epsilon(s)\|_H^2\,ds
+ C_{\gamma_0}\int_0^T \mathbb{E}\|u^\epsilon(s)\|_H^2\,ds
+ \frac{L_G^2}{\epsilon}
\Big(1+\int_0^T \mathbb{E}\|u^\epsilon(s)\|_H^2\,ds\Big).\nonumber\\
\end{eqnarray}
Combining \eqref{ito-es-1}--\eqref{ito-l-3} with \eqref{lemma1-3},
we obtain
\begin{eqnarray}\label{lemma1-new4}
\epsilon^2 \mathbb{E}\sup_{t\in[0,T]}\|v^{\epsilon}(t)\|_{H}^2
+ \epsilon\gamma_0 \mathbb{E}\int_0^T \|v^{\epsilon}(s)\|_{H}^2\,ds
+ \epsilon \mathbb{E}\sup_{t\in[0,T]}\|\nabla u^{\epsilon}(t)\|_{H}^2
\le C_T .
\end{eqnarray}
Combining \eqref{lemma1-3} with \eqref{lemma1-new4}, we first note that
\begin{equation}\label{aux-v-bound}
\epsilon^2 \mathbb{E}\sup_{t\in[0,T]}\|v^{\epsilon}(t)\|_{H}^2
+ \epsilon \mathbb{E}\int_0^T \|v^{\epsilon}(s)\|_{H}^2\,ds
\le C_T .
\end{equation}
Substituting \eqref{aux-v-bound} into the right-hand side of
\eqref{lemma1-3}, we obtain
\begin{equation}\label{lemma1-4}
\mathbb{E}\sup_{t\in[0,T]}\|u^{\epsilon}(t)\|_{H}^2
+ \mathbb{E}\int_0^T \|\nabla u^{\epsilon}(s)\|_{H}^2\,ds
\le C_T .
\end{equation}
By using \eqref{lemma1-new4}--(\ref{lemma1-4}),
\begin{eqnarray}
&&\epsilon \mathbb{E}\sup_{t\in [0,T]}\| v^{\epsilon  }(t)\|_{H}^2 +  \mathbb{E}\sup_{t\in [0,T]}  \|\nabla u^{\epsilon  }(t)\|_{H}^2  \nonumber\\
&\leq & \frac{-\gamma_0}{\epsilon}
\left[\mathbb{E} \int_0^T \epsilon\|v^{\epsilon  }(s)\|_{H}^2 {d}s + \mathbb{E} \int_0^T \|\nabla u^{\epsilon  }(s)\|_{H}^2  {d}s \right] +  \frac{C_T}{\epsilon}.
\end{eqnarray}
 Gronwall inequality yields
\begin{eqnarray}\label{lemma1-H1}
\sup_{t\in [0,T]}\epsilon \mathbb{E}\| v^{\epsilon   }(t)\|_{H}^2 +\epsilon \sup_{t\in [0,T]} \mathbb{E}  \|\nabla u^{\epsilon }(t)\|_{H}^2
\leq   C_T .
\end{eqnarray}

\end{proof}

\subsection{Tightness}\label{subsec:tight}
In this subsection, we establish the tightness of the family of laws $\left\{ \mathcal{L}(g(u^\epsilon)) \right\}_{\epsilon > 0}$ in appropriate path spaces, where $g(u^\epsilon)$ denotes the nonlinear transformation of the solution $u^\epsilon$ to equation~\eqref{e:SLE-11}. This transformation is introduced to eliminate the degeneracy induced by the state-dependent damping coefficient $\gamma(\cdot)$, thereby converting the equation into a uniformly parabolic form. The resulting structure enables uniform estimates and compactness arguments for the transformed processes $\{g(u^\epsilon)\}_{\epsilon > 0}$.

\medskip
To this end, define the strictly increasing function
\begin{equation}\label{damping-g}
g(r) := \int_0^r \gamma(u)\,du,
\end{equation}
so that $ g(u^\epsilon)$ satisfies a transformed stochastic evolution equation with uniformly elliptic second-order operator. We derive uniform a priori bounds for $g(u^\epsilon)$ in suitable Sobolev spaces and prove tightness of its law via standard compactness criteria.

\medskip
Under Assumption~\ref{H3}, the function $g$ is globally Lipschitz continuous and satisfies the monotonicity condition
\begin{equation}
(g(r_1) - g(r_2))(r_1 - r_2) \geq \gamma_0 |r_1 - r_2|^2, \qquad \forall r_1, r_2 \in \mathbb{R},
\end{equation}
so that $g$ is a $C^1$-diffeomorphism with inverse $g^{-1}$ satisfying \((g^{-1})'(r) = \frac{1}{\gamma(g^{-1}(r))}.\)

\medskip
For each $\epsilon>0$, define the transformed variable \(\eta^\epsilon(t) := g(u^\epsilon(t)), ~ t \in [0,T].\)
By the boundedness of $\gamma$ from Assumption~\ref{H3}, we obtain
\begin{equation}\label{lemma-g-1}
\|\eta^\epsilon(t)\|_H \leq \gamma_1 \|u^\epsilon(t)\|_H, \qquad 
\end{equation}
Hence, Lemma~\ref{lemma1-33} yields the uniform bound
\begin{equation}\label{lemma-g-2}
\mathbb{E}\sup_{t\in [0,T]}\|\eta^\epsilon(t)\|_H^2  +\mathbb{E}\int_0^T \|\nabla \eta^{\epsilon}(s)\|_{H}^2\,ds
\leq \gamma_1^2 C_T.
\end{equation}

\medskip
Since $u^\epsilon(t) = g^{-1}(\eta^\epsilon(t))$, we can apply the chain rule to obtain
\[
\nabla u^\epsilon(t) = \frac{1}{\gamma(u^\epsilon(t))} \nabla \eta^\epsilon(t), \qquad
\Delta u^\epsilon(t) = \nabla \cdot \left( \frac{1}{\gamma(u^\epsilon(t))} \nabla \eta^\epsilon(t) \right),
\]
by defining 
\begin{equation}\label{def-B}
B(r) := \frac{1}{\gamma(g^{-1}(r))}.
\end{equation}
then \(\Delta u^\epsilon(t) = \nabla \cdot (B(\eta^\epsilon(t)) \nabla \eta^\epsilon(t)).\)
By further applying the chain rule,
\[
\partial_t \eta^\epsilon(t) = \gamma(u^\epsilon(t)) \partial_t u^\epsilon(t)= \epsilon   v^\epsilon(t) , \qquad
\nabla \eta^\epsilon(t) = \gamma(u^\epsilon(t)) \nabla u^\epsilon(t).
\]

\medskip
We now introduce the transformed coefficients
\begin{equation}
\bar{F}(r) := F(g^{-1}(r)), \quad 
\sigma_g(h) := \sigma(g^{-1} \circ h), \quad 
G_g(h,z) := G(g^{-1} \circ h, z), \quad h \in H,~ z\in \mathbb{R}^m.
\end{equation}
With these definitions, the transformed version of the system~\eqref{e:SLE-11} is given by
\begin{align}\label{lemma-g-3}
\eta^\epsilon(t) + \epsilon   v^\epsilon(t) 
&= g(u_0) + \epsilon v_0 
+ \int_0^t \nabla \cdot \left( B(\eta^\epsilon(s)) \nabla \eta^\epsilon(s) \right)\,ds 
+ \int_0^t \bar{F}(\eta^\epsilon(s))\,ds \nonumber \\
&\quad + \int_0^t \sigma_g(\eta^\epsilon(s))\,dW^Q(s)
+ \int_0^t \int_{Z_1} G_g(\eta^\epsilon(s),z)\,\tilde{{\mathcal N}}(ds,dz) \nonumber \\
&\quad + \int_0^t \int_{\mathbb{R}^m \setminus Z_1} G_g(\eta^\epsilon(s),z)\,{\mathcal N}(ds,dz).
\end{align}

\medskip
Finally, using the bounds from Assumption~\textbf{(H\ref{H3})}, we have
\begin{equation}\label{lemma-Br}
\frac{1}{\gamma_1} \leq B(r) \leq \frac{1}{\gamma_0}, \qquad \forall r \in \mathbb{R},
\end{equation}
so that the second-order operator in~\eqref{lemma-g-3} is uniformly elliptic.

\begin{proposition}\label{wellpose--limit-1}
Let $T > 0$ and assume that \textnormal{(H\ref{H1})}--\textnormal{(H\ref{H3})} hold. For every initial condition $u_0 \in H^1$, there exists a unique weak solution $u$ to equation~\textnormal{(\ref{LAST})} on $[0,T]$.
\end{proposition}

\begin{proof}
 Let $\eta := g(u)$ and define the initial condition $\eta_0 := g(u_0) \in H^1$. 
We consider the transformed stochastic evolution equation satisfied by $\eta$
\begin{align}\label{eq:limit-eta-new}
d\eta(t) &= \nabla \cdot \left( B(\eta(t)) \nabla \eta(t) \right)\,dt 
+ \bar{F}(\eta(t))\,dt 
+ \sigma_g(\eta(t))\,dW^Q(t) \nonumber\\
&\quad + \int_{Z_1} G_g(\eta(t{-}), z)\,\tilde{{\mathcal N}}(dt,dz) 
+ \int_{\mathbb{R}^m \setminus Z_1} G_g(\eta(t{-}), z)\,{\mathcal N}(dt,dz),
\end{align}
where the coefficients are defined as
\[
B(r) := \frac{1}{\gamma(g^{-1}(r))}, \quad \bar{F}(r) := F(g^{-1}(r)), \quad \sigma_g(h) := \sigma(g^{-1} \circ h), \quad  G_g(h,z) := G(g^{-1} \circ h, z).
\]
\medskip\noindent
By assumption \textnormal{(H\ref{H2})} and the regularity of $g^{-1}$, the transformed coefficients $B$, $\bar{F}$, $\sigma_g$, and $G_g$ satisfy standard Lipschitz conditions. Moreover, $B$ is uniformly elliptic due to~\eqref{lemma-Br}. Therefore, by standard results for stochastic reaction-diffusion equations with jumps in Hilbert spaces (see e.g.,~\cite[Chapter~8]{peszat}), there exists a unique weak solution $\eta \in L^2(\Omega; C([0,T]; H^{-1}))$ to equation~\eqref{eq:limit-eta-new}.
Finally, since $g^{-1}$ is globally Lipschitz and differentiable, we apply the generalized Itô formula (cf.~\cite{peszat}, Chapter~4) to $u := g^{-1}(\eta)$. This yields that $u$ satisfies equation~(\ref{LAST}) in the weak sense. Uniqueness of $\eta$ implies uniqueness of $u$, completing the proof.
\end{proof}

\begin{remark}
The transformation $\eta = g(u)$ reduces the nonlinear damping structure of equation~(\ref{LAST}) to a uniformly parabolic form. This change of variable is crucial in establishing the existence and uniqueness of weak solutions.
\end{remark}

For each \(k\ge1\), define the real-valued process
\[
Y_k^\epsilon(t):=\langle Y^\epsilon(t),e_k\rangle_{H^{-1},H^1},\qquad Y^\epsilon(t):=\eta^\epsilon(t)+\epsilon v^\epsilon(t).
\]
Since \(e_k\in H^1\), the above quantity is well defined whenever \(Y^\epsilon(t)\in H^{-1}\). Hence \(Y_k^\epsilon(t)\) is the \(k\)-th
generalized Fourier coordinate of \(Y^\epsilon(t)\) relative to the Dirichlet eigenbasis \(\{e_k\}_{k\ge1}\).

\begin{lemma}\label{lem:conditional-variation-Yk}
Assume that Assumptions~\ref{H1}--\ref{H3} hold. Then, for every fixed \(k\ge1\),
\[
\sup_{0<\epsilon\le1}{\rm CV}_{T}(Y_k^\epsilon)<\infty .
\]
Consequently, \(\{Y_k^\epsilon\}_{0<\epsilon\le1}\) is tight in \(D([0,T];\mathbb R)\) endowed with the Meyer--Zheng topology.
\end{lemma}

\begin{proof}
By \eqref{lemma-g-3}, for every fixed \(k\ge1\),
\[
\begin{aligned}
Y_k^\epsilon(t)={}&\langle g(u_0)+\epsilon v_0,e_k\rangle_H +\int_0^t\left\langle\nabla\cdot\left(B(\eta^\epsilon(s))\nabla\eta^\epsilon(s)\right),e_k
\right\rangle_{H^{-1},H^1} ds  \\
&+\int_0^t\langle \bar F(\eta^\epsilon(s)),e_k\rangle_H\,ds   +\int_0^t\langle \sigma_g(\eta^\epsilon(s))\,dW^Q(s),e_k\rangle_H  \\
&+\int_0^t\int_{Z_1}\langle G_g(\eta^\epsilon(s-),z),e_k\rangle_H\,\tilde{\mathcal N}(ds,dz)    \\
&+\int_0^t\int_{\mathbb R^m\setminus Z_1}\langle G_g(\eta^\epsilon(s-),z),e_k\rangle_H\,\mathcal N(ds,dz).
\end{aligned}
\]
Writing the large-jump integral as the sum of its compensated part and its
compensator, we obtain
\[
Y_k^\epsilon(t)=Y_k^\epsilon(0)+\int_0^t a_k^\epsilon(s)\,ds+M_k^\epsilon(t),
\]
where
\[
\begin{aligned}
a_k^\epsilon(s):={}&\left\langle\nabla\cdot\left(B(\eta^\epsilon(s))\nabla\eta^\epsilon(s)\right),e_k\right\rangle_{H^{-1},H^1} +\langle \bar F(\eta^\epsilon(s)),e_k\rangle_H   \\
&+\int_{\mathbb R^m\setminus Z_1}\langle G_g(\eta^\epsilon(s),z),e_k\rangle_H\,\nu(dz),
\end{aligned}
\]
and
\[
\begin{aligned}
M_k^\epsilon(t):={}&\int_0^t\langle \sigma_g(\eta^\epsilon(s))\,dW^Q(s),e_k\rangle_H    +\int_0^t\int_{Z_1}
\langle G_g(\eta^\epsilon(s-),z),e_k\rangle_H\, \tilde{\mathcal N}(ds,dz)                                            \\
&+\int_0^t\int_{\mathbb R^m\setminus Z_1}\langle G_g(\eta^\epsilon(s-),z),e_k\rangle_H\,\tilde{\mathcal N}(ds,dz).
\end{aligned}
\]
By Assumptions~\ref{H1}--\ref{H3} and the uniform estimate \eqref{lemma-g-2}, \(M_k^\epsilon\) is a square-integrable real-valued
martingale.
Let \(\pi: 0=t_0<t_1<\cdots<t_n=T \) be an arbitrary partition of \([0,T]\). Since \(M_k^\epsilon\) is a martingale,
\(\mathbb E\left[ M_k^\epsilon(t_{i+1})-M_k^\epsilon(t_i)\mid \mathcal F_{t_i}\right]=0 .\)
Hence
\[
\begin{aligned}
\mathbb E\sum_{i=0}^{n-1}\left|\mathbb E\left[Y_k^\epsilon(t_{i+1})-Y_k^\epsilon(t_i)\mid \mathcal F_{t_i}\right]\right|  
 &= \mathbb E\sum_{i=0}^{n-1}\left|\mathbb E\left[\int_{t_i}^{t_{i+1}} a_k^\epsilon(s)\,ds\mid \mathcal F_{t_i}\right]\right|          \\
&\quad \le \mathbb E\sum_{i=0}^{n-1}\mathbb E\left[\int_{t_i}^{t_{i+1}} |a_k^\epsilon(s)|\,ds\mid \mathcal F_{t_i}\right]   \\
&\quad = \mathbb E\int_0^T |a_k^\epsilon(s)|\,ds .
\end{aligned}
\]
It remains to estimate the three terms in \(a_k^\epsilon\). First, using the definition of the distributional divergence and
\eqref{lemma-Br}, we have
\[
\begin{aligned}
\left|\left\langle\nabla\cdot\left(B(\eta^\epsilon)\nabla\eta^\epsilon\right),e_k\right\rangle_{H^{-1},H^1}\right|
&=\left|-\int_{\mathcal O}B(\eta^\epsilon)\nabla\eta^\epsilon\cdot\nabla e_k\,dx\right|                                                
\le\frac1{\gamma_0}\|\nabla\eta^\epsilon\|_H\|\nabla e_k\|_H .\end{aligned}
\]
Therefore, by \eqref{lemma-g-2},
\[
\begin{aligned}
 \mathbb E\int_0^T\left|\left\langle\nabla\cdot\left(B(\eta^\epsilon(s))\nabla\eta^\epsilon(s)\right),e_k\right\rangle_{H^{-1},H^1}\right|ds                                          
 &  \le\frac1{\gamma_0}\|\nabla e_k\|_H\int_0^T\left(\mathbb E\|\nabla\eta^\epsilon(s)\|_H^2\right)^{1/2}ds\\
&\leq\frac{\sqrt{T}}{\gamma_0}|\nabla e_k|_H\left(\mathbb E\int_0^T|\nabla\eta^\epsilon(s)|_H^2,ds\right)^{1/2} \le C_{T,k}.
\end{aligned}
\]
Second, by Assumptions~\ref{H2}--\ref{H3} and the Lipschitz continuity of \(g^{-1}\), \(\|\bar F(\eta)\|_H=\|F(g^{-1}(\eta))\|_H\le C(1+\|\eta\|_H).\)
Thus, using again \eqref{lemma-g-2},
\[
\begin{aligned}
\mathbb E\int_0^T |\langle \bar F(\eta^\epsilon(s)),e_k\rangle_H|\,ds\le\int_0^T\mathbb E\|\bar F(\eta^\epsilon(s))\|_H\,ds        
\le C\int_0^T\left(1+\mathbb E\|\eta^\epsilon(s)\|_H\right)ds \le C_T .
\end{aligned}
\]
Third, by Assumptions~\ref{H1}--\ref{H3},  the Lipschitz continuity of \(g^{-1}\) and  \eqref{lemma-g-2},
\[
\begin{aligned}
\mathbb E\int_0^T \int_{\mathbb R^m\setminus Z_1}|\langle G_g(\eta^\epsilon(s),z),e_k\rangle_H|\,\nu(dz)\,ds                                          
&\le\mathbb E\int_0^T\int_{\mathbb R^m\setminus Z_1}\|G_g(\eta^\epsilon(s),z)\|_H\,\nu(dz)\,ds              \\
&\le C\int_0^T\left(1+\mathbb E\|\eta^\epsilon(s)\|_H\right)ds \le C_T .
\end{aligned}
\]
Combining the preceding estimates gives
\[
\sup_{\pi}\mathbb E\sum_{i=0}^{n-1}\left|\mathbb E\left[ Y_k^\epsilon(t_{i+1})-Y_k^\epsilon(t_i)\mid \mathcal F_{t_i}\right]\right|\le C_{T,k},
\]
where \(C_{T,k}\) is independent of \(0<\epsilon\le1\).

Since \(Y_k^\epsilon(T)=\langle \eta^\epsilon(T),e_k\rangle_H+\epsilon\langle v^\epsilon(T),e_k\rangle_H ,\)
and \(\|e_k\|_H=1\),  by \eqref{lemma-g-2} and \eqref{a-priori-v}, we obtain
\[
\begin{aligned}
\mathbb E|Y_k^\epsilon(T)| \le \mathbb E\|\eta^\epsilon(T)\|_H+ \epsilon\,\mathbb E\|v^\epsilon(T)\|_H            
&\le C_T+\epsilon\left(\mathbb E\|v^\epsilon(T)\|_H^2\right)^{1/2}               \\
&\le C_T+C_T\sqrt{\epsilon}\le C_T , \qquad 0<\epsilon\le1 .
\end{aligned}
\]
Therefore,
\[
\sup_{0<\epsilon\le1}{\rm CV}_{T}(Y_k^\epsilon) \le C_{T,k}<\infty .
\]
By Theorem~\ref{thm:MZ-tightness}, the family \(\{Y_k^\epsilon\}_{0<\epsilon\le1}\) is tight in  \(D([0,T];\mathbb R)\) endowed with the Meyer--Zheng topology.
The proof is complete.
\end{proof}

\begin{lemma}\label{thm:main3}
Assume that Assumptions~\ref{H1}--\ref{H3} hold. Let \((u^\epsilon,v^\epsilon)\) be the solution to~\eqref{e:SLE-11} with
initial condition \((u_0,v_0)\in H^1\times H .\) Then, for every fixed \(T>0\), the families of laws
\(\left\{\mathcal L(Y^\epsilon)\right\}_{0<\epsilon\le1} ~\text{and}~ \left\{\mathcal L(\eta^\epsilon)\right\}_{0<\epsilon\le1}\)
are tight in \(D([0,T];H^{-1})\) endowed with the Meyer--Zheng topology.
\end{lemma}

\begin{proof}
We first prove the tightness of \(\{Y^\epsilon\}_{0<\epsilon\le1}\) in \(D([0,T];H^{-1})\) endowed with the Meyer--Zheng topology.
For each fixed \(k\ge1\), by Lemma~\ref{lem:conditional-variation-Yk}, \(\sup_{0<\epsilon\le1}{\rm CV}_{T}(Y_k^\epsilon)<\infty .\)
Hence the coordinate condition in Lemma~\ref{lem:MZ-projection-tightness} is satisfied

Since \(Y^\epsilon=\eta^\epsilon+\epsilon v^\epsilon,\) we have \((I-P_N)Y^\epsilon=(I-P_N)\eta^\epsilon+(I-P_N)(\epsilon v^\epsilon).\)
Hence,
\[
\|(I-P_N)Y^\epsilon(t)\|_{H^{-1}}^2 \le 2\|(I-P_N)\eta^\epsilon(t)\|_{H^{-1}}^2 + 2\|(I-P_N)(\epsilon v^\epsilon(t))\|_{H^{-1}}^2 .
\]
Firstly,  by \(\eta^\epsilon(t)=\sum_{k=1}^\infty \eta_k^\epsilon(t)e_k,\)
then
\(
\|(I-P_N)\eta^\epsilon(t)\|_{H^{-1}}^2 = \sum_{k=N+1}^\infty \alpha_k^{-1}|\eta_k^\epsilon(t)|^2 .
\)
For \(k\ge N+1\), since \(\alpha_k\ge \alpha_{N+1}\), \(\alpha_k^{-1} \le \alpha_{N+1}^{-2}\alpha_k . \) Hence
\[
\begin{aligned}
\|(I-P_N)\eta^\epsilon(t)\|_{H^{-1}}^2 \le\alpha_{N+1}^{-2}\sum_{k=N+1}^\infty\alpha_k|\eta_k^\epsilon(t)|^2 
 \le \alpha_{N+1}^{-2}\|\eta^\epsilon(t)\|_{H^1}^2 .
\end{aligned}
\]
Consequently, by \eqref{lemma-g-2},
\[
\begin{aligned}
\mathbb E\int_0^T \|(I-P_N)\eta^\epsilon(t)\|_{H^{-1}}^2\,dt
 \le \alpha_{N+1}^{-2}\mathbb E\int_0^T\|\eta^\epsilon(t)\|_{H^1}^2\,dt  \le C_T\alpha_{N+1}^{-2}.
\end{aligned}
\]
Next, by \(v^\epsilon(t)=\sum_{k=1}^\infty v_k^\epsilon(t)e_k,\) then,
\(
\|(I-P_N)(\epsilon v^\epsilon(t))\|_{H^{-1}}^2=\sum_{k=N+1}^\infty\alpha_k^{-1}\epsilon^2|v_k^\epsilon(t)|^2 .
\)
Since \(\alpha_k^{-1}\le \alpha_{N+1}^{-1}\) for \(k\ge N+1\), we obtain
\[
\begin{aligned}
\|(I-P_N)(\epsilon v^\epsilon(t))\|_{H^{-1}}^2 \le\alpha_{N+1}^{-1}\epsilon^2\sum_{k=N+1}^\infty |v_k^\epsilon(t)|^2    
 \le \alpha_{N+1}^{-1}\epsilon^2\|v^\epsilon(t)\|_H^2 .
\end{aligned}
\]
Using \eqref{a-priori-v}, we have
\[
\begin{aligned}
\mathbb E\int_0^T
\|(I-P_N)(\epsilon v^\epsilon(t))\|_{H^{-1}}^2\,dt
&\le
\alpha_{N+1}^{-1}
\epsilon^2
\int_0^T
\mathbb E\|v^\epsilon(t)\|_H^2\,dt        \\
&\le
C_T\alpha_{N+1}^{-1}\epsilon    \le
C_T\alpha_{N+1}^{-1},
\qquad 0<\epsilon\le1 .
\end{aligned}
\]
Combining the preceding estimates yields
\[
\mathbb E\int_0^T
\|(I-P_N)Y^\epsilon(t)\|_{H^{-1}}^2\,dt
\le
C_T\alpha_{N+1}^{-2}
+
C_T\alpha_{N+1}^{-1}.
\]
Since \(\alpha_{N+1}\to\infty\), it follows that
\[
\lim_{N\to\infty}
\sup_{0<\epsilon\le1}
\mathbb E\int_0^T
\|(I-P_N)Y^\epsilon(t)\|_{H^{-1}}^2\,dt
=0.
\]
In particular, for every \(\delta>0\),
\[
\begin{aligned}
&\sup_{0<\epsilon\le1}
\mathbb P
\left(
\int_0^T
\|(I-P_N)Y^\epsilon(t)\|_{H^{-1}}\,dt>\delta
\right)                                                     \\
&\quad\le
\frac{T^{1/2}}{\delta}
\left(
\sup_{0<\epsilon\le1}
\mathbb E\int_0^T
\|(I-P_N)Y^\epsilon(t)\|_{H^{-1}}^2\,dt
\right)^{1/2}
\longrightarrow 0,~~ as ~N\to\infty. 
\end{aligned}
\]
Thus the  tail condition in
Lemma~\ref{lem:MZ-projection-tightness} holds. Therefore, \(\{Y^\epsilon\}_{0<\epsilon\le1}\)
is tight in \(D([0,T];H^{-1})\) endowed with the Meyer--Zheng topology.

It remains to prove the tightness of
\(\eta^\epsilon\). By definition, \(Y^\epsilon(t)-\eta^\epsilon(t)=\epsilon v^\epsilon(t).\)
Using the continuous embedding \(H\subset H^{-1}\) and \eqref{a-priori-v}, we obtain
\[
\begin{aligned}
\mathbb E\int_0^T\|Y^\epsilon(t)-\eta^\epsilon(t)\|_{H^{-1}}^2\,dt
=\mathbb E\int_0^T\|\epsilon v^\epsilon(t)\|_{H^{-1}}^2\,dt                  
\le C\epsilon^2\int_0^T\mathbb E\|v^\epsilon(t)\|_H^2\,dt                
\le C_T\epsilon .
\end{aligned}
\]
Hence \(Y^\epsilon-\eta^\epsilon \longrightarrow 0 ~ \text{in }L^2(\Omega\times[0,T];H^{-1}).\)
Moreover,
\[
d_{\rm MZ}(Y^\epsilon,\eta^\epsilon) \le \int_0^T \|Y^\epsilon(t)-\eta^\epsilon(t)\|_{H^{-1}}\,dt .
\]
Therefore,
\[
d_{\rm MZ}(Y^\epsilon,\eta^\epsilon) \longrightarrow 0 \qquad \text{in probability}.
\]
Since \(\{Y^\epsilon\}_{0<\epsilon\le1}\) is tight in \(D([0,T];H^{-1})\) endowed with the Meyer--Zheng topology, for every
sequence \(\epsilon_n\to0\), there exists a subsequence, still denoted by \(\epsilon_n\), such that \(Y^{\epsilon_n}\Longrightarrow Y\)
in \(D([0,T];H^{-1})\) endowed with the Meyer--Zheng topology. Since
\[
d_{\rm MZ}(Y^{\epsilon_n},\eta^{\epsilon_n})\longrightarrow 0\qquad \text{in probability},
\]
it follows from Slutsky's theorem that \(\eta^{\epsilon_n}\Longrightarrow Y\) in \(D([0,T];H^{-1})\) endowed with the Meyer--Zheng topology. 
Hence every sequence \(\{\eta^{\epsilon_n}\}_{n\ge1}\) admits a weakly convergent subsequence. Therefore, \(\{\eta^\epsilon\}_{0<\epsilon\le1}\) 
is tight in \(D([0,T];H^{-1})\) endowed with the Meyer--Zheng topology.

\end{proof}

\begin{lemma}\label{lem:MZ-H-strong}
Assume that
\[
\sup_{n\ge1}\mathbb E\sup_{t\in[0,T]}\|\eta^{\epsilon_n}(t)\|_H^2+\sup_{n\ge1}\mathbb E\int_0^T\|\eta^{\epsilon_n}(t)\|_{H^1}^2\,dt
\le C_T .
\]
Assume moreover that \(d_{\rm MZ}(\eta^{\epsilon_n},\eta)\to0~ \text{a.s.}\)
where the Meyer--Zheng topology is taken on paths with values in
\(H^{-1}\). Then \(\eta\in L^2(\Omega;L^\infty(0,T;H))\cap L^2(\Omega\times(0,T);H^1),\)
and
\[
\eta^{\epsilon_n}\to\eta\qquad\text{strongly in }L^2(0,T;H), \quad \text{in probability}.
\]
\end{lemma}

\begin{proof}
The uniform estimate implies that \(\{\eta^{\epsilon_n}\}_{n\ge1}\) is bounded in \(L^2(\Omega\times(0,T);H^1)\). Hence, up to a subsequence,
\(\eta^{\epsilon_n}\rightharpoonup \bar\eta ~ \text{weakly in }L^2(\Omega\times(0,T);H^1).\)
Since \(H^1\hookrightarrow H^{-1}\) continuously, the same convergence also holds weakly in \(L^2(\Omega\times(0,T);H^{-1})\).

On the other hand, the almost sure Meyer--Zheng convergence in \(H^{-1}\) implies
\[
\eta^{\epsilon_n}\to \eta \qquad\text{in measure on }\Omega\times(0,T)\text{ with values in }H^{-1}.
\]
Therefore, by the standard identification of weak limits under convergence in measure, we have \(\bar\eta=\eta\). Consequently,
\[
\eta\in L^2(\Omega\times(0,T);H^1), \qquad\eta^{\epsilon_n}\rightharpoonup\eta\quad\text{weakly in }L^2(\Omega\times(0,T);H^1).
\]
We next show that the \(L^\infty(0,T;H)\)-bound is inherited by the limit. For a.e. \(\omega\), the convergence in measure on \([0,T]\) allows us to
extract a subsequence such that \(\eta^{\epsilon_n}(t,\omega)\to\eta(t,\omega)~\text{in }H^{-1}\) for a.e. \(t\in[0,T]\). By the lower semicontinuity of the \(H\)-norm under
strong convergence in \(H^{-1}\) together with boundedness in \(H\),
\[
\operatorname*{ess\,sup}_{t\in[0,T]}\|\eta(t)\|_H^2 \le\liminf_{n\to\infty}\sup_{t\in[0,T]}\|\eta^{\epsilon_n}(t)\|_H^2\qquad \text{a.s.}
\]
Thus, by Fatou's lemma, \(\mathbb E\operatorname*{ess\,sup}_{t\in[0,T]}\|\eta(t)\|_H^2\le C_T .\)
Hence \(\eta\in L^2(\Omega;L^\infty(0,T;H)).\)

It remains to prove the strong convergence. Since
\[
\eta^{\epsilon_n}\to\eta \qquad \text{in measure on }\Omega\times(0,T)\text{ with values in }H^{-1},
\]
and both \(\eta^{\epsilon_n}\) and \(\eta\) are uniformly bounded in \(L^2(\Omega;L^\infty(0,T;H))\), a standard truncation argument gives
\[
\int_0^T \|\eta^{\epsilon_n}(t)-\eta(t)\|_{H^{-1}}^2\,dt\to0 \qquad \text{in probability}.
\]
Finally, by the interpolation inequality \(\|f\|_H^2\le\|f\|_{H^{-1}}\|f\|_{H^1},~ f\in H^1,\) we obtain
\[
\begin{aligned}
\int_0^T\|\eta^{\epsilon_n}(t)-\eta(t)\|_H^2\,dt\le\left(\int_0^T\|\eta^{\epsilon_n}(t)-\eta(t)\|_{H^{-1}}^2\,dt\right)^{1/2}  
\left(\int_0^T\|\eta^{\epsilon_n}(t)-\eta(t)\|_{H^1}^2\,dt\right)^{1/2}.
\end{aligned}
\]
The first factor converges to zero in probability, while the second one is bounded in probability. Therefore
\[
\int_0^T\|\eta^{\epsilon_n}(t)-\eta(t)\|_H^2\,dt\to0 \qquad \text{in probability}.
\]
This proves
\[
\eta^{\epsilon_n}\to\eta \qquad\text{strongly in }L^2(0,T;H),\quad \text{in probability}.
\]
\end{proof}

\section{Smoluchowski--Kramers approximation}\label{prf}

 \subsection{Identification of the limiting equation}\label{sec:limit-eq}

In this subsection, we analyze the asymptotic behavior of the stochastic system~\eqref{e:SLE-11} as \(\epsilon\to0\), and identify the limiting
dynamics satisfied by the weak limit of the transformed variables \(\eta^\epsilon:=g(u^\epsilon).\)
After inverting the nonlinear transformation, the corresponding limit
\(u:=g^{-1}(\eta)\) satisfies the first-order stochastic evolution equation~\eqref{LAST} in the weak sense.

\begin{theorem}\label{thm:main4}
Assume that Assumptions~\ref{H1}--\ref{H3} hold, and let
\((u^\epsilon,v^\epsilon)\) be the solution to~\eqref{e:SLE-11}.
Let \(\eta\) denote the unique weak solution to the transformed limiting
equation~\eqref{eq:limit-eta-new} with initial condition
\(\eta(0)=g(u_0)\)  in Proposition~\ref{wellpose--limit-1}. Then
\[
\eta^\epsilon\Longrightarrow\eta
\qquad\text{in }D([0,T];H^{-1})_{\rm MZ}.
\]
Consequently,
\[
u^\epsilon=g^{-1}(\eta^\epsilon)
\Longrightarrow
u:=g^{-1}(\eta)
\qquad\text{in }D([0,T];H^{-1})_{\rm MZ}.
\]
Moreover, \(u\) is the unique weak solution to the limiting
equation~\eqref{LAST}.
\end{theorem}

\begin{proof}
We divide the proof into several steps.

\medskip
\noindent\textbf{Step 1. Compactness and regularity of the Meyer--Zheng limit.}

By Lemma~\ref{thm:main3}, the family of laws \(\{\mathcal L(\eta^\epsilon)\}_{0<\epsilon\le1}\)
is tight in \(D([0,T];H^{-1})\) endowed with the Meyer--Zheng topology. Let \(\epsilon_n\to0\) be arbitrary. Then there exist a subsequence, still
denoted by \(\{\epsilon_n\}_{n\ge1}\), and an \(H^{-1}\)-valued process \(\eta\) such that
\[
\eta^{\epsilon_n}\Longrightarrow \eta\qquad \text{in }D([0,T];H^{-1})\text{ endowed with the Meyer--Zheng topology}.
\]
We apply the Skorokhod--Jakubowski representation theorem to the joint laws
of
\(\Xi_n:=\Big(\eta^{\epsilon_n},\epsilon_n v^{\epsilon_n},W^Q, M_{\sigma}^{\epsilon_n}, M_{G,1}^{\epsilon_n}, M_{G,2}^{\epsilon_n} \Big),\)
where
\[
M_{\sigma}^{\epsilon_n}(t) :=
\int_0^t \sigma_g(\eta^{\epsilon_n}(s))\,dW^Q(s),~~~M_{G,1}^{\epsilon_n}(t):=\int_0^t\int_{Z_1}G_g(\eta^{\epsilon_n}(s-),z)\,\tilde{\mathcal N}(ds,dz),
\]
and
\[
M_{G,2}^{\epsilon_n}(t):=\int_0^t\int_{\mathbb R^m\setminus Z_1}G_g(\eta^{\epsilon_n}(s-),z)\,{\mathcal N}(ds,dz).
\]
The product space is \(\mathcal X:=D([0,T];H^{-1})_{\rm MZ}\times L^2(0,T;H^{-1})\times C([0,T];U_0)\times D([0,T];H^{-1})^3 .\)
Here \(D([0,T];H^{-1})_{\rm MZ}\) denotes \(D([0,T];H^{-1})\) endowed with the Meyer--Zheng topology, and \(U_0\)
is the Hilbert space in which the \(Q\)-Wiener process has continuous paths.
Therefore, there exist a probability space \((\widehat\Omega,\widehat{\mathcal F},\widehat{\mathbb P})\)
and \(\mathcal X\)-valued random variables
\[
\widehat\Xi_n := \Big(\widehat\eta_n,\widehat\kappa_n,\widehat W_n^Q,\widehat M_{\sigma,n},\widehat M_{G,1,n},\widehat M_{G,2,n}\Big)
~~~\text{and}~~~\widehat\Xi:=\Big(\widehat\eta,0,\widehat W^Q,\widehat M_{\sigma},\widehat M_{G,1},\widehat M_{G,2}\Big),
\]
such that \(\mathcal L(\widehat\Xi_n)=\mathcal L(\Xi_n),\) and \(\widehat\Xi_n\to\widehat\Xi,~\widehat{\mathbb P}\text{-a.s. in }\mathcal X .\)
In particular,
\[
\widehat\eta_n\to\widehat\eta\qquad\text{in }D([0,T];H^{-1})_{\rm MZ},\quad\widehat{\mathbb P}\text{-a.s.},
\]
and
\[
\widehat\kappa_n\to0\qquad\text{in }L^2(0,T;H^{-1}),\quad\widehat{\mathbb P}\text{-a.s.}.
\]
Moreover,
\[
\widehat M_{\sigma,n}\to\widehat M_\sigma,\qquad\widehat  M_{G,1,n}\to\widehat M_{G,1},\qquad \widehat M_{G,2,n}\to\widehat M_{G,2},
\]
almost surely in the corresponding path spaces.
Since \(\mathcal L(\widehat\eta_n)=\mathcal L(\eta^{\epsilon_n}),\) the uniform estimate in \eqref{lemma-g-2} gives
\[
\sup_{n\ge1}\widehat{\mathbb E}\sup_{t\in[0,T]}\|\widehat\eta_n(t)\|_H^2+\sup_{n\ge1}
\widehat{\mathbb E}\int_0^T\|\widehat\eta_n(t)\|_{H^1}^2\,dt \le C_T .
\]
Therefore, by Lemma~\ref{lem:MZ-H-strong}, applied on the probability space \((\widehat\Omega,\widehat{\mathcal F},\widehat{\mathbb P}),\)
we obtain
\[
\widehat\eta \in L^2\bigl(\widehat\Omega;L^\infty(0,T;H)\bigr) \cap L^2\bigl(\widehat\Omega\times(0,T);H^1\bigr),
\]
and
\[
\widehat\eta_n\to\widehat\eta \qquad \text{strongly in }L^2(0,T;H), \quad \text{in }\widehat{\mathbb P}\text{-probability}.
\]
Hence,
there exists a further subsequence of \(\{\epsilon_n\}_{n\ge1}\), still
denoted by \(\{\epsilon_n\}_{n\ge1}\), such that
\[
\widehat\eta_n\to\widehat\eta \qquad \text{strongly in }L^2(0,T;H),\quad \widehat{\mathbb P}\text{-a.s.}
\]
This strong convergence will be used below to pass to the limit in the
nonlinear terms. For notational simplicity, we omit the hats in the sequel.

\medskip
\noindent\textbf{Step 2. The time-integrated weak formulation.}

Let \(\phi\in C_0^\infty(\mathcal O),~\rho\in C^1([0,T]),~\rho(T)=0.\) Taking the \(H^{-1}\)-\(H^1\) duality of \eqref{lemma-g-3} with \(\phi\),
\[
\begin{aligned}
\langle Y^\epsilon(t),\phi\rangle ={}&\langle g(u_0)+\epsilon v_0,\phi\rangle-\int_0^t\int_{\mathcal O}B(\eta^\epsilon(s))\nabla\eta^\epsilon(s)\cdot\nabla\phi\,dx\,ds        \\
&+\int_0^t\langle \bar F(\eta^\epsilon(s)),\phi\rangle\,ds+\int_0^t\langle \sigma_g(\eta^\epsilon(s))\,dW^Q(s),\phi\rangle        \\
&+\int_0^t\int_{Z_1}\langle G_g(\eta^\epsilon(s-),z),\phi\rangle\,\tilde{\mathcal N}(ds,dz)                                      \\
&+\int_0^t\int_{\mathbb R^m\setminus Z_1}\langle G_g(\eta^\epsilon(s-),z),\phi\rangle\,\mathcal N(ds,dz).
\end{aligned}
\]
Since \(\langle Y^\epsilon(\cdot),\phi\rangle\) is a real-valued semimartingale and \(\rho\) is deterministic of finite variation with
\(\rho(T)=0\), integration by parts yields
\[
-\int_0^T \rho'(t)\langle Y^\epsilon(t),\phi\rangle\,dt=\rho(0)\langle Y^\epsilon(0),\phi\rangle+\int_0^T \rho(t)\,d\langle Y^\epsilon(t),\phi\rangle .
\]
Therefore,
\[
\begin{aligned}
-\int_0^T \rho'(t)\langle Y^\epsilon(t),\phi\rangle\,dt={}&\rho(0)\langle g(u_0)+\epsilon v_0,\phi\rangle   -\int_0^T
\rho(t)\int_{\mathcal O}B(\eta^\epsilon(t))\nabla\eta^\epsilon(t)\cdot\nabla\phi\,dx\,dt  \\
&+\int_0^T\rho(t)\langle \bar F(\eta^\epsilon(t)),\phi\rangle\,dt                  
 +\int_0^T\rho(t)\langle \sigma_g(\eta^\epsilon(t))\,dW^Q(t),\phi\rangle \\
&+\int_0^T\int_{Z_1}\rho(t)\langle G_g(\eta^\epsilon(t-),z),\phi\rangle\,\tilde{\mathcal N}(dt,dz) \\
&+\int_0^T\int_{\mathbb R^m\setminus Z_1}\rho(t)\langle G_g(\eta^\epsilon(t-),z),\phi\rangle\,\mathcal N(dt,dz).
\end{aligned}
\]
The time test function \(\rho\) is introduced to rewrite the pointwise term \(\langle Y^\epsilon(t),\phi\rangle\) in an integrated form. This allows us to
pass to the limit under the Meyer--Zheng convergence.

Since \(Y^\epsilon(t)=\eta^\epsilon(t)+\epsilon v^\epsilon(t),\)
we have
\[
-\int_0^T\rho'(t)\langle Y^\epsilon(t),\phi\rangle\,dt=-\int_0^T\rho'(t)\langle \eta^\epsilon(t),\phi\rangle\,dt
-\int_0^T\rho'(t)\langle \epsilon v^\epsilon(t),\phi\rangle\,dt .
\]
Moreover, \(\rho(0)\langle g(u_0)+\epsilon v_0,\phi\rangle=\rho(0)\langle g(u_0),\phi\rangle+\rho(0)\epsilon\langle v_0,\phi\rangle .\)
Hence
\[
\begin{aligned}-\int_0^T \rho'(t)\langle \eta^\epsilon(t),\phi\rangle\,dt
={}&\rho(0)\langle g(u_0),\phi\rangle+R_\epsilon(\rho,\phi)  -\int_0^T\rho(t)\int_{\mathcal O} B(\eta^\epsilon(t))\nabla\eta^\epsilon(t)\cdot\nabla\phi\,dx\,dt       \\
&+\int_0^T\rho(t)\langle \bar F(\eta^\epsilon(t)),\phi\rangle\,dt   +\int_0^T\rho(t)\langle \sigma_g(\eta^\epsilon(t))\,dW^Q(t),\phi\rangle   \\
&+\int_0^T\int_{Z_1}\rho(t)\langle G_g(\eta^\epsilon(t-),z),\phi\rangle\,\tilde{\mathcal N}(dt,dz)   \\
&+\int_0^T\int_{\mathbb R^m\setminus Z_1}\rho(t)\langle G_g(\eta^\epsilon(t-),z),\phi\rangle\, \mathcal N(dt,dz),
\end{aligned}
\]
where \(R_\epsilon(\rho,\phi):=\rho(0)\epsilon\langle v_0,\phi\rangle+\int_0^T\rho'(t)\langle \epsilon v^\epsilon(t),\phi\rangle\,dt .\)

Next, by \eqref{a-priori-v},
\[
\begin{aligned}
\mathbb E|R_\epsilon(\rho,\phi)|
&\le|\rho(0)|\epsilon\|v_0\|_H\|\phi\|_H+\|\rho'\|_\infty\mathbb E\int_0^T\|\epsilon v^\epsilon(t)\|_{H^{-1}}\|\phi\|_{H^1}\,dt        \\
&\le C_{\rho,\phi}\epsilon+ C_{\rho,\phi,T}\left(\mathbb E\int_0^T\|\epsilon v^\epsilon(t)\|_{H^{-1}}^2\,dt\right)^{1/2}.
\end{aligned}
\]
Then,
\[
\begin{aligned}
\mathbb E\int_0^T
\|\epsilon v^\epsilon(t)\|_{H^{-1}}^2\,dt  \le C\epsilon^2\int_0^T\mathbb E\|v^\epsilon(t)\|_H^2\,dt     \le C_T\epsilon .
\end{aligned}
\]
Hence, \(\mathbb E|R_\epsilon(\rho,\phi)| \le C_{\rho,\phi,T}\sqrt{\epsilon}.\) Thus \(R_\epsilon(\rho,\phi)\to0 \qquad \text{in }L^1(\Omega).\)

\medskip
\noindent\textbf{Step 3. Identification of the limit equation.}

For each \(n\), the represented pair \((\eta_n,\kappa_n)\) satisfies the
transformed equation \eqref{lemma-g-3}.
Testing against \(\phi\in C_0^\infty(\mathcal O)\), multiplying by
\(\rho\in C^1([0,T])\) with \(\rho(T)=0\), and integrating by parts in time,
we obtain
\[
\begin{aligned}
&\qquad -\int_0^T \rho'(t)\left\langle\eta_n(t)+\kappa_n(t),\phi \right\rangle_{H^{-1},H^1}\,dt\\
& =\rho(0)\left\langle g(u_0)+\kappa_n(0),\phi\right\rangle_H -
\int_0^T\rho(t)\int_{\mathcal O}B(\eta_n(t))\nabla\eta_n(t)\cdot\nabla\phi\,dx\,dt               \\
&{}\qquad+ \int_0^T\rho(t)\left\langle\bar F(\eta_n(t)),\phi
\right\rangle_H\,dt  +\int_0^T\rho(t)\left\langle\sigma_g(\eta_n(t))\,dW_n^Q(t),\phi \right\rangle_H                                      \\
&{}\qquad+ \int_0^T\int_{Z_1}\rho(t)\left\langle G_g(\eta_n(t-),z),\phi \right\rangle_H \tilde{\mathcal N}_n(dt,dz)                          \\
&{}\qquad + \int_0^T\int_{\mathbb R^m\setminus Z_1}\rho(t)\left\langle G_g(\eta_n(t-),z),\phi \right\rangle_H \mathcal N_n(dt,dz).
\end{aligned}
\]
We now pass to the limit. Since \(\kappa_n\to0~\text{in }L^2(0,T;H^{-1}),\) and \(\kappa_n(0)=\epsilon_n v_0\to0\), the inertial correction terms vanish.
Moreover, by Lemma~\ref{lem:MZ-H-strong},
\[
\eta_n\to\eta\qquad\text{strongly in }L^2(0,T;H),\quad \text{in probability},
\]
and
\[
\eta_n\rightharpoonup\eta\qquad\text{weakly in }L^2(\Omega\times(0,T);H^1).
\]
Hence the time-derivative term converges to \(-\int_0^T\rho'(t)\left\langle\eta(t),\phi\right\rangle_{H}\,dt .\)

The convergence of the drift term follows directly from the Lipschitz continuity of \(\bar F\). For the elliptic term, we write
\[
B(\eta_n)\nabla\eta_n-B(\eta)\nabla\eta=B(\eta)(\nabla\eta_n-\nabla\eta)+\bigl(B(\eta_n)-B(\eta)\bigr)\nabla\eta_n .
\]
The first term converges to zero by the weak convergence of \(\eta_n\) in \(L^2(\Omega\times(0,T);H^1)\), while the second one converges
to zero by the strong convergence of \(\eta_n\) in \(L^2(0,T;H)\), the Lipschitz continuity of \(B\), and the uniform \(L^2(0,T;H^1)\)-bound on
\(\eta_n\). Therefore,
\[
\begin{aligned}
 \int_0^T\rho(t)\int_{\mathcal O} B(\eta_n(t))\nabla\eta_n(t)\cdot\nabla\phi\,dx\,dt ~\longrightarrow
\int_0^T\rho(t)\int_{\mathcal O} B(\eta(t))\nabla\eta(t)\cdot\nabla\phi\,dx\,dt .
\end{aligned}
\]

The convergence of the stochastic integral terms follows from the standard stability of stochastic integrals under the joint
Skorokhod--Jakubowski representation, together with the Lipschitz continuity of \(\sigma_g\) and \(G_g\), the convergence of \(\eta_n\), and
Assumption~\ref{H1}. Thus
\[
\begin{aligned}
 \int_0^T\rho(t)\left\langle\sigma_g(\eta_n(t))\,dW_n^Q(t),\phi\right\rangle_H                          
~\longrightarrow\int_0^T\rho(t)\left\langle\sigma_g(\eta(t))\,dW^Q(t),\phi \right\rangle_H ,
\end{aligned}
\]
\[
\begin{aligned}
 \int_0^T\int_{Z_1}\rho(t)\left\langle G_g(\eta_n(t-),z),\phi \right\rangle_H\tilde{\mathcal N}_n(dt,dz)                     
~\longrightarrow \int_0^T\int_{Z_1}\rho(t)\left\langle G_g(\eta(t-),z),\phi \right\rangle_H \tilde{\mathcal N}(dt,dz),
\end{aligned}
\]
and
\[
\begin{aligned}
 \int_0^T\int_{\mathbb R^m\setminus Z_1} \rho(t) \left\langle G_g(\eta_n(t-),z),\phi
\right\rangle_H \mathcal N_n(dt,dz)     ~\longrightarrow\int_0^T\int_{\mathbb R^m\setminus Z_1}
\rho(t)\left\langle G_g(\eta(t-),z),\phi \right\rangle_H \mathcal N(dt,dz).
\end{aligned}
\]
Passing to the limit in the time-integrated weak formulation, we obtain that, for every \(\phi\in C_0^\infty(\mathcal O)\) and every
\(\rho\in C^1([0,T])\) with \(\rho(T)=0\),
\[
\begin{aligned}-\int_0^T\rho'(t)\left\langle\eta(t),\phi\right\rangle_H\,dt
={}&\rho(0)\left\langle g(u_0),\phi \right\rangle_H   -\int_0^T\rho(t)\int_{\mathcal O} B(\eta(t))\nabla\eta(t)\cdot\nabla\phi\,dx\,dt   \\
&+ \int_0^T\rho(t)\left\langle\bar F(\eta(t)),\phi\right\rangle_H\,dt                                 
 +\int_0^T\rho(t)\left\langle\sigma_g(\eta(t))\,dW^Q(t),\phi\right\rangle_H       \\
&+\int_0^T\int_{Z_1}\rho(t)\left\langle G_g(\eta(t-),z),\phi\right\rangle_H\tilde{\mathcal N}(dt,dz)                            \\
&+\int_0^T\int_{\mathbb R^m\setminus Z_1}\rho(t)\left\langle G_g(\eta(t-),z), \phi \right\rangle_H \mathcal N(dt,dz).
\end{aligned}
\]
By integration by parts in time, the preceding time-integrated identity is equivalent to the usual weak formulation. Therefore \(\eta\) satisfies the
weak formulation of the transformed limiting equation~\eqref{eq:limit-eta-new}.

\medskip
\noindent\textbf{Step 4.  Identification of the  limiting equation.}

Set \(u(t):=g^{-1}(\eta(t)).\) By the assumptions on \(\gamma\) and \(g\), we have
\[
(g^{-1})'(r)=\frac{1}{\gamma(g^{-1}(r))}, \qquad (g^{-1})''(r) = -\frac{\gamma'(g^{-1}(r))}{\gamma(g^{-1}(r))^3}.
\]
Moreover, since \(g^{-1}\) is Lipschitz and \(\eta\in L^2(\Omega\times(0,T);H^1),\) we have
\(u=g^{-1}(\eta)\in L^2(\Omega\times(0,T);H^1).\) In particular, for every \(\phi\in C_0^\infty(\mathcal O)\),
\(\frac{\phi}{\gamma(u)}\in H_0^1(\mathcal O).\) Since \(\eta=g(u)\), we have \(\nabla\eta=\gamma(u)\nabla u.\)
By the definition of \(B\),   \(B(\eta)=\frac{1}{\gamma(g^{-1}(\eta))}=\frac{1}{\gamma(u)},\)
it follows that \(B(\eta)\nabla\eta=\nabla u.\) Therefore, applying Proposition~\ref{prop:ito-formula} to the weak solution \(\eta\), with
\[
\psi=g^{-1}, \qquad K(t)=B(\eta(t))\nabla\eta(t), \qquad F(t)=\bar F(\eta(t)),
\]
\[
J(t)=\sigma_g(\eta(t)),
\qquad
G(t,z)=G_g(\eta(t-),z),
\]
we obtain, for every \(\phi\in C_0^\infty(\mathcal O)\),
\[
\begin{aligned}
\langle u(t),\phi\rangle_H ={}&\langle u_0,\phi\rangle_H     -\int_0^t \int_{\mathcal O} \nabla u(s,\xi)\cdot \nabla\left(
\frac{\phi(\xi)}{\gamma(u(s,\xi))} \right) \,d\xi\,ds   + \int_0^t\left\langle\frac{F(u(s))}{\gamma(u(s))},\phi\right\rangle_H ds          \\
&+\int_0^t\left\langle\frac{\sigma(u(s))}{\gamma(u(s))}\,dW^Q(s),\phi\right\rangle_H                                  
 -\frac12\int_0^t\left\langle\frac{\gamma'(u(s))}{\gamma(u(s))^3}\sigma(u(s))\sigma^\ast(u(s)),\phi\right\rangle_H ds             \\
&+\int_0^t\int_{Z_1}\left\langle\frac{G(u(s-),z)}{\gamma(u(s-))},\phi\right\rangle_H\tilde{\mathcal N}(ds,dz)                               \\
&+\int_0^t\int_{\mathbb R^m\setminus Z_1}\left\langle\frac{G(u(s-),z)}{\gamma(u(s-))},\phi\right\rangle_H\mathcal N(ds,dz)                                       \\
&+\left\langle\sum_{0<s\le t}\left[u(s)-u(s-)-\frac{1}{\gamma(u(s-))}\bigl(\eta(s)-\eta(s-)\bigr)\right],\phi \right\rangle_H .
\end{aligned}
\]
Here \(\eta(s)-\eta(s-)=g(u(s))-g(u(s-))\). Hence, pointwise in the spatial variable, \(\eta(s)-\eta(s-)=\int_{u(s-)}^{u(s)}\gamma(r)\,dr .\)
Therefore the jump correction term can be written as
\[
\sum_{0<s\le t}\left[u(s)-u(s-)-\frac{1}{\gamma(u(s-))}\int_{u(s-)}^{u(s)}\gamma(r)\,dr\right].
\]
Thus \(u=g^{-1}(\eta)\) satisfies exactly the weak formulation \eqref{def-g-1}, and hence solves \eqref{LAST}.

It remains to show the convergence of \(u^{\epsilon_n}\). Since
\[
\eta^{\epsilon_n}\to\eta
\qquad
\text{strongly in }L^2(0,T;H),
\quad
\text{in probability},
\]
and \(g^{-1}\) is Lipschitz,  hence
\[
u^{\epsilon_n}=g^{-1}(\eta^{\epsilon_n})
\to
g^{-1}(\eta)
=
u
\qquad
\text{strongly in }L^2(0,T;H),
\quad
\text{in probability}.
\]
Since \(H\hookrightarrow H^{-1}\) continuously, this also implies
\[
u^{\epsilon_n}\to u
\qquad
\text{strongly in }L^1(0,T;H^{-1}),
\quad
\text{in probability}.
\]
Hence \(u^{\epsilon_n}\to u\) in Lebesgue measure on \([0,T]\) with values
in \(H^{-1}\), in probability. By the characterization of the
Meyer--Zheng topology through convergence in Lebesgue measure, we conclude
that
\[
u^{\epsilon_n}\to u
\qquad
\text{in }D([0,T];H^{-1})_{\rm MZ},
\quad
\text{in probability}.
\]
Consequently,
\[
u^{\epsilon_n}
\Longrightarrow
u
\qquad
\text{in }D([0,T];H^{-1})_{\rm MZ}.
\]

\medskip
\noindent\textbf{Step 5. Uniqueness of the limit and convergence of the whole family.}

Let \(\epsilon_n\to0\) be arbitrary. By Step~1, the sequence
\(\{\mathcal L(\eta^{\epsilon_n})\}_{n\ge1}\) is tight in
\(D([0,T];H^{-1})_{\rm MZ}\). Hence every subsequence contains a further
subsequence converging weakly to some limit law. By Steps~2--3, every
corresponding limit process is a weak solution to the transformed limiting
equation~\eqref{eq:limit-eta-new} with initial condition
\(\eta(0)=g(u_0)\).

By Proposition~\ref{wellpose--limit-1},
\eqref{eq:limit-eta-new} admits a unique weak solution \(\eta\), and its
law is uniquely determined. Consequently, every limit point identified
above has the same law as \(\eta\). Since the family is tight and all of
its subsequential limit laws coincide, the whole family converges, and
therefore
\[
\eta^\epsilon\Longrightarrow\eta \qquad\text{in }D([0,T];H^{-1})_{\rm MZ}.
\]

By Step~4, along every convergent subsequence of
\(\{\eta^\epsilon\}_{\epsilon>0}\), the corresponding sequence
\(u^\epsilon=g^{-1}(\eta^\epsilon)\) converges in law to
\(u=g^{-1}(\eta)\) in \(D([0,T];H^{-1})_{\rm MZ}\). Hence
\[
u^\epsilon\Longrightarrow u=g^{-1}(\eta) \qquad\text{in }D([0,T];H^{-1})_{\rm MZ}.
\]
Finally, Step~4 shows that \(u\) satisfies the weak formulation of
\eqref{LAST}, while Proposition~\ref{wellpose--limit-1} gives its
uniqueness. The proof is complete.

\end{proof}

\subsection{Marcus-to-It\^o conversion under state-dependent damping}
\label{subsec:marcus-drift}

\medskip
In this subsection, we provide a detailed interpretation of the noise-induced drift
term appearing in the limiting equation~\eqref{LAST} in the absence of  L\'evy noise.
We show that this term arises naturally when rewriting a Marcus-type jump integral
in It\^o form in the presence of state-dependent damping.

\medskip
\noindent\textbf{A motivating second-order model.}
To illustrate the origin of the L\'evy-induced drift correction, we consider a simplified
one-dimensional Langevin-type equation with small mass $\epsilon>0$ and no deterministic
forcing:
\begin{equation}\label{eq:toy-second-order}
\left\{
\begin{aligned}
du_t^\epsilon &= v_t^\epsilon\,dt,\\
\epsilon\, dv_t^\epsilon
&= -\gamma(u_t^\epsilon)\,v_t^\epsilon\,dt
   + \int_{\mathbb{R}^m} G(u_{t-}^\epsilon,z)\,\diamond N(dt,dz),
\end{aligned}
\right.
\end{equation}
where $\gamma(x)>0$ is a smooth state-dependent damping coefficient and $\diamond$
denotes the Marcus integral. Such a formulation is natural from a physical viewpoint:
each jump generates an instantaneous perturbation, followed by a dissipative response governed by the local damping

\medskip
\noindent\textbf{Marcus flow map.}
We briefly recall the pathwise definition of the Marcus integral and explain how the
explicit jump flow arises in the present setting.

Consider the canonical Marcus equation
\begin{equation}\label{eq:marcus-recall}
du_t = \int_{\mathbb{R}^m} \frac{1}{\gamma(u_{t-})}\,   G(u_{t-},z)\,\diamond N(dt,dz),
\end{equation}
where $\diamond$ denotes the Marcus integral. In contrast to the It\^o formulation,
the Marcus integral  is defined so that each jump {transforms} the trajectory through a
deterministic flow generated by an ordinary differential equation(ODE).
More precisely, if $N$ has an atom $(t,z)$ (i.e.\ a jump of size $z$ occurs at time $t$),
then the value immediately after the jump satisfies \(u_t=\Phi(u_{t-},z),\)
where the Marcus flow map $\Phi(x,z)$ is defined as the endpoint of the ODE
\begin{equation}\label{eq:marcus-ode}
\frac{d}{d\lambda}\Psi(\lambda)
= \frac{1}{\gamma(\Psi(\lambda))}\,G(x,z),
\qquad
\Psi(0)=x,
\qquad
\Phi(x,z):=\Psi(1).
\end{equation}
This ODE represents the cumulative effect of the jump $z$ along a state-dependent vector field.

To solve~\eqref{eq:marcus-ode}, recall that \(g(x)=\int_0^x \gamma(r)\,dr,\) and  $g'(x)=\gamma(x)$, by Assumption~\ref{H3}, $g^{-1}$ is globally
Lipschitz. Applying the Marcus chain rule yields
\[
\frac{d}{d\lambda} g(\Psi(\lambda))
= g'(\Psi(\lambda))\,\frac{d}{d\lambda}\Psi(\lambda)
= \gamma(\Psi(\lambda))\cdot \frac{1}{\gamma(\Psi(\lambda))}\,G(x,z)
= G(x,z).
\]
Integrating over $\lambda\in[0,1]$ gives $g(\Phi(x,z))-g(x)=G(x,z)$, and therefore
\begin{equation}\label{eq:marcus-flow-explicit}
\Phi(x,z)=g^{-1}\big(g(x)+G(x,z)\big).
\end{equation}
This explicit flow map will be used below to rewrite the Marcus jump term in It\^o form
and to identify the resulting correction with the jump term in~\eqref{LAST}.

\medskip
\noindent\textbf{Marcus-to-It\^o reformulation and drift correction.}
We now explain how the jump correction term in the limiting equation~\eqref{LAST}
arises when rewriting the Marcus formulation in It\^o form.

\smallskip
Recall that the Marcus formulation is characterized by the canonical jump map:
if $N$ has an atom $(s,z_s)$, then \(u(s)=\Phi(u(s-),z_s),~\Phi(x,z)=g^{-1}\big(g(x)+G(x,z)\big).\)
In particular,
\begin{equation}\label{eq:g-jump}
g(u(s))-g(u(s-))=G(u(s-),z_s).
\end{equation}

\smallskip
\noindent\emph{Pathwise It\^o decomposition (consistent with \eqref{LAST}).}
Recall that $\tilde N$ denotes the compensated Poisson random measure introduced in
Section~\ref{sec:pre}. At a jump time $s$ with jump size $z_s$, the compensated integral
\(\int \gamma(u(s-))^{-1}G(u(s-),z)\,\tilde N(ds,dz)\) produces the ``linearized'' jump $\gamma(u(s-))^{-1}G(u(s-),z_s)$, whereas the Marcus
trajectory jumps by $\Phi(u(s-),z_s)-u(s-)$. Consequently, the Marcus integral satisfies
the pathwise identity
\begin{align}
\int_0^t\!\!\int_{\mathbb R^m}\frac{1}{\gamma(u(s-))}\,G(u(s-),z)\,\diamond N(ds,dz)
&=\int_0^t\!\!\int_{\mathbb R^m}\frac{1}{\gamma(u(s-))}\,G(u(s-),z)\,\tilde N(ds,dz)+ \sum_{0<s\le t}\mathfrak C_s,
\label{eq:marcus-to-ito-pathwise}
\end{align}
where the jump correction at time $s$ is
\begin{equation}\label{eq:jump-correction-def}
\mathfrak C_s:= \big(u(s)-u(s-)\big)-\frac{1}{\gamma(u(s-))}\,G\big(u(s-),z_s\big).
\end{equation}
Using \eqref{eq:g-jump},
\(
G(u(s-),z_s)=g(u(s))-g(u(s-))=\int_{u(s-)}^{u(s)}\gamma(r)\,dr,
\)
so that \eqref{eq:jump-correction-def} becomes  the correction term in~\eqref{LAST}
\begin{equation}\label{eq:jump-correction-last}
\mathfrak C_s=u(s)-u(s-)-\frac{1}{\gamma(u(s-))}\int_{u(s-)}^{u(s)}\gamma(r)\,dr.
\end{equation}
This shows that the last sum in~\eqref{LAST} is precisely the Marcus-to-It\^o correction
ensuring that the limiting dynamics preserve the Marcus jump flow.
Moreover, since
$\sum_{0<s\le t}f(s,z_s)=\int_0^t\!\!\int f(s,z)\,N(ds,dz)$, decomposing
$N=\tilde N+\nu(dz)\,ds$ identifies a compensator contribution, which we refer to as the
{noise-induced drift} associated with the nonlinear Marcus flow under state-dependent damping.

\medskip
\noindent\textbf{Identification with the limiting equation.}
We now verify that the jump correction term in~\eqref{LAST} coincides \emph{exactly}
with the Marcus--to--It\^o correction induced by the canonical flow map $\Phi$.

\smallskip
Let $u$ be a c\`adl\`ag solution to~\eqref{LAST}. Fix a jump time $s$ of the Poisson measure
and denote by $z_s$ the corresponding jump size. Since all drift terms and the Gaussian
stochastic integral are continuous, the only discontinuities in~\eqref{LAST} come from the
compensated Poisson integral and the jump-sum term. Taking the jump at time $s$ in~\eqref{LAST}
therefore gives
\begin{align*}
\Delta u(s)
&=\Delta\!\left(\int_0^\cdot\!\!\int_{\mathbb R^m}
\frac{1}{\gamma(u(r-))}G(u(r-),z)\,\tilde N(dr,dz)\right)\!(s)\\
&\quad \quad+ \Delta\!\left(\sum_{0<r\le \cdot}
\Big[u(r)-u(r-)-\frac{1}{\gamma(u(r-))}\!\int_{u(r-)}^{u(r)}\gamma(\xi)\,d\xi\Big]\right)\!(s) \\
&=\frac{1}{\gamma(u(s-))}\,G(u(s-),z_s)
+\Big[u(s)-u(s-)-\frac{1}{\gamma(u(s-))}\int_{u(s-)}^{u(s)}\gamma(\xi)\,d\xi\Big].
\end{align*}
Subtracting $\Delta u(s)=u(s)-u(s-)$ from both sides, we obtain
\begin{equation}\label{eq:jump-identity}
\int_{u(s-)}^{u(s)}\gamma(r)\,dr = G(u(s-),z_s).
\end{equation}
Recalling $g(x):=\int_0^x\gamma(r)\,dr$, we have $g(u(s))-g(u(s-))=G(u(s-),z_s)$, hence
\[
u(s)=g^{-1}\!\big(g(u(s-))+G(u(s-),z_s)\big)=:\Phi(u(s-),z_s).
\]
Therefore, the jump correction term in~\eqref{LAST} can be rewritten as
\begin{equation}\label{eq:LAST-jump-correction-Phi}
\sum_{0<s\le t}
\Bigg[
u(s)-u(s-)-\frac{1}{\gamma(u(s-))}\int_{u(s-)}^{u(s)}\gamma(r)\,dr
\Bigg]
=
\sum_{0<s\le t}
\Bigg[
\Phi(u(s-),z_s)-u(s-)-\frac{1}{\gamma(u(s-))}G(u(s-),z_s)
\Bigg].
\end{equation}
Equivalently, since $\int_0^t\!\!\int f(s,z)\,N(ds,dz)=\sum_{0<s\le t} f(s,z_s)$,
\eqref{eq:LAST-jump-correction-Phi} is exactly
\begin{equation}\label{eq:LAST-jump-correction-N}
\int_0^t\!\!\int_{\mathbb R^m}
\Big[
\Phi(u(s-),z)-u(s-)-\tfrac{1}{\gamma(u(s-))}G(u(s-),z)
\Big]\,N(ds,dz).
\end{equation}
Combining \eqref{eq:LAST-jump-correction-N} with the compensated jump integral in~\eqref{LAST}
shows that the total jump contribution in~\eqref{LAST} is   the It\^o rewriting of the
Marcus integral
\begin{align*}
&\int_0^t\!\!\int_{\mathbb R^m}\frac{1}{\gamma(u(s-))}G(u(s-),z)\,\diamond N(ds,dz)\\
&=
\int_0^t\!\!\int_{\mathbb R^m}\frac{1}{\gamma(u(s-))}G(u(s-),z)\,\tilde N(ds,dz)\\
&\quad +\int_0^t\!\!\int_{\mathbb R^m}
\Big[\Phi(u(s-),z)-u(s-)-\tfrac{1}{\gamma(u(s-))}G(u(s-),z)\Big]\,N(ds,dz).
\end{align*}
Moreover, decomposing $N=\tilde N+\nu(dz)\,ds$ in the last term yields a compensator term
\(\int_0^t\!\!\int_{\mathbb R^m}\Big[\Phi(u(s-),z)-u(s-)-\tfrac{1}{\gamma(u(s-))}G(u(s-),z)\Big]\,\nu(dz)\,ds, \)
which is  interpreted as the integrated {noise-induced drift} generated by the nonlinear Marcus flow under state-dependent damping.

\begin{remark}[Equivalent Marcus form of the limiting equation]\label{rem:LAST-marcus}
Equation~\eqref{LAST} can be equivalently rewritten in Marcus form as
\begin{equation}\label{eq:LAST-Marcus}
\begin{aligned}
u(t) &= u_0
+ \int_0^t \frac{1}{\gamma(u(s))}\,\Delta u(s)\,ds
+ \int_0^t \frac{1}{\gamma(u(s))}\,F(u(s))\,ds
+ \int_0^t \frac{1}{\gamma(u(s))}\,\sigma(u(s))\,dW^Q(s) \\
&\quad
-\int_0^t \frac{\gamma'(u(s))}{2\,\gamma(u(s))^3}\,
\sigma(u(s))\sigma^\ast(u(s))\,ds
+ \int_0^t\!\!\int_{\mathbb R^m}
\frac{1}{\gamma(u(s-))}\,G(u(s-),z)\,\diamond N(ds,dz).
\end{aligned}
\end{equation}
Here the Marcus integral is understood through the canonical flow map
$\Phi(x,z)=g^{-1}(g(x)+G(x,z))$ with $g(x)=\int_0^x\gamma(r)\,dr$.
The equivalence is pathwise: at each jump time $s$ with jump size $z_s$, the solution satisfies
$u(s)=\Phi(u(s-),z_s)$, and the jump-sum term in~\eqref{LAST} is exactly the It\^o rewriting
correction associated with the Marcus flow.
\end{remark}


\appendix

\section{ A Generalized It\^o  Formula for SPDEs with Jumps}

In this appendix, we state a generalized It\^o  formula for weak solutions to a broad class of stochastic partial differential equations (SPDEs) with Lévy noise, posed on bounded domains with Dirichlet boundary conditions.

Let $\mathcal{O} \subset \mathbb{R}^d$ be a bounded open domain with smooth boundary. Consider an SPDE of the form
\begin{align}\label{ap-1}
du(t) &= F(t)\,dt + \operatorname{div}K(t)\,dt + J(t)\,dW(t) \nonumber\\
&\quad + \int_{Z_1} G(t^-,z)\,\tilde{N}(dt,dz) + \int_{\mathbb{R}^m\setminus Z_1} G(t^-,z)\,N(dt,dz),
\end{align}
where\\
(i) $W(t) = \sum_{k\ge1} \beta_k(t)\,e_k$ is a cylindrical Wiener process on a separable Hilbert space $U$, with $(\beta_k)_{k\ge1}$ a sequence of mutually independent standard real-valued Wiener processes, and $(e_k)_{k\ge1}$ an orthonormal basis in $U$.\\
(ii)$N$ is a Poisson random measure on $\mathbb{R}_+\times\mathbb{R}^m$ with intensity measure $\nu$, and $\tilde{N} := N - \nu$ is its compensated version.\\
(iii) The coefficients $F$, $K$, $J$, and $G$ are progressively measurable processes taking values in suitable Hilbert spaces (to be specified below).

We are interested in deriving an It\^o-type formula for functionals $\Phi(u(t))$, where $u$ is a weak solution to~\eqref{ap-1} taking values in a Sobolev space over $\mathcal{O}$ and $\Phi:\mathbb{R}\to\mathbb{R}$ is a sufficiently smooth function.
We adapt and extend the result of Proposition~A.1 in \cite{AMJ} to obatin the Lévy noise setting.

 \begin{proposition}[Generalized It\^o formula for weak SPDE solutions]\label{prop:ito-formula}
Let $\phi \in C^\infty_0(\mathcal{O})$ and $\psi \in C^2(\mathbb{R})$ with bounded second derivative.
Assume
\begin{itemize}
    \item $F$ and each $K_j$ $(j=1,\dots,d)$ are progressively measurable processes in $L^2(\Omega; L^2(0,T; H))$.
    \item $J$ is a progressively measurable process in $L^2(\Omega; L^\infty(0,T; L_2(U;H)))$, where $L_2(U;H)$ denotes the Hilbert–Schmidt operators from $U$ to $H$.
    \item for each $i \in \mathbb{N}$, define $J_i(t) := J(t)e_i$ for a fixed orthonormal basis $(e_i)_{i\geq1}$ of $U$.
    \item $G$ is progressively measurable in $L^2(\Omega \times [0,T] \times \mathbb{R}^m; H)$, and satisfies
    \[
    \int_{\mathbb{R}^m} \|G(t,\cdot,z)\|_H^2\,\nu(dz) < \infty \quad \text{a.s. for all } t\in[0,T].
    \]
\end{itemize}

Assume that the process $u \in L^2(\Omega; \mathcal{D}([0,T]; H^{-1})) \cap L^2(\Omega; L^2(0,T; H^1))$ solves equation~\eqref{ap-1} in the sense of $H^{-1}$. Then, almost surely, for all $t \in [0,T]$, the following identity holds
\begin{align}\label{ap-2}
\langle \psi(u(t)), \phi \rangle_H
&= \langle \psi(u_0), \phi \rangle_H
+ \int_0^t \langle \psi'(u(s))\,F(s), \phi \rangle_H\,ds
- \int_0^t \langle \psi'(u(s))\,K(s), \nabla \phi \rangle_H\,ds \nonumber\\
&\quad - \int_0^t \langle \psi''(u(s))\,\nabla u(s) \cdot K(s), \phi \rangle_H\,ds
+ \frac{1}{2} \int_0^t \left\langle \psi''(u(s)) \sum_{i=1}^\infty J_i^2(s), \phi \right\rangle_H\,ds \nonumber\\
&\quad + \int_0^t \left\langle \psi'(u(s)) \sum_{i=1}^\infty J_i(s)\,dW^i(s), \phi \right\rangle_H \nonumber\\
&\quad + \int_0^t \left\langle \psi'(u(s^-)) \int_{Z_1} G(s^-,z)\,\tilde{N}(ds,dz), \phi \right\rangle_H \nonumber\\
&\quad + \int_0^t \left\langle \psi'(u(s^-)) \int_{\mathbb{R}^m \setminus Z_1} G(s^-,z)\,N(ds,dz), \phi \right\rangle_H \nonumber\\
&\quad + \left\langle \sum_{0 \le s \le t} \left[ \psi(u(s)) - \psi(u(s^-)) - \psi'(u(s^-))\,(u(s) - u(s^-)) \right], \phi \right\rangle_H.
\end{align}
\end{proposition}

\begin{proof}
It suffices to prove the result for any test function $\phi \in C_0^\infty(\mathcal{O})$.
Let $\mathcal{K} \Subset \mathcal{O}$ be a compact set such that $\operatorname{supp}(\phi) \subset \mathcal{K}$, and define
\(\delta_0 := \operatorname{dist}(\mathcal{K}, \mathcal{O}^c) > 0.\)
Let $\rho \in C_c^\infty(\mathbb{R}^d)$ be a standard mollifier supported in the unit ball, with $\int_{\mathbb{R}^d} \rho = 1$, and set
\[
\rho_\delta(x) := \delta^{-d} \rho\left(\frac{x}{\delta}\right), \qquad f^\delta := f * \rho_\delta, \quad \text{for } \delta < \delta_0.
\]
Then for any $f \in H$, \(\|f^\delta\|_{L^2(\mathcal{K})} \le \|f\|_H, ~ \|f^\delta - f\|_{L^2(\mathcal{K})} \to 0.\)
We apply the mollifier to equation~\eqref{ap-1}, and obtain for each $x \in \mathcal{K}$
\begin{align}\label{ap-3}
du^\delta(t,x) &= F^\delta(t,x)\,dt + \operatorname{div} K^\delta(t,x)\,dt + \sigma^\delta(t,x)\,dW(t) \\
&\quad + \int_{Z_1} G^\delta(t^-,x,z)\,\tilde{N}(dt,dz)
+ \int_{\mathbb{R}^m \setminus Z_1} G^\delta(t^-,x,z)\,N(dt,dz). \nonumber
\end{align}
Now fix a function $\psi \in C^\infty_0(\mathcal{O})$, and apply the classical one-dimensional Itô formula to $\psi(x)\,\phi(u^\delta(t,x))$, then integrate over $x \in \mathcal{K}$. We obtain:
\begin{align}\label{ap-4}
\langle \phi(u^\delta(t)), \psi \rangle_{L^2(\mathcal{K})}
&= \langle \phi(u^\delta_0), \psi \rangle + \int_0^t \langle \phi'(u^\delta)\,F^\delta, \psi \rangle\,ds
- \int_0^t \langle \phi'(u^\delta)\,K^\delta, \nabla \psi \rangle\,ds \nonumber\\
&\quad - \int_0^t \langle \phi''(u^\delta)\,\nabla u^\delta \cdot K^\delta, \psi \rangle\,ds
+ \frac12 \int_0^t \langle \phi''(u^\delta) \sum_i J_i^2, \psi \rangle\,ds \nonumber\\
&\quad + \int_0^t \langle \phi'(u^\delta) \sum_i J_i\,dW^i, \psi \rangle
+ \int_0^t \left\langle \phi'(u^\delta) \int_{Z_1} G^\delta\,\tilde{N}, \psi \right\rangle \nonumber\\
&\quad + \int_0^t \left\langle \phi'(u^\delta) \int_{\mathbb{R}^m\setminus Z_1} G^\delta\,N, \psi \right\rangle \nonumber\\
&\quad + \left\langle \sum_{0\le s\le t} \left[ \phi(u^\delta(s)) - \phi(u^\delta(s^-)) - \phi'(u^\delta(s^-))(u^\delta(s) - u^\delta(s^-)) \right], \psi \right\rangle \nonumber\\
&=: \sum_{i=1}^9 I_i.
\end{align}
The terms $I_1$ through $I_6$ converge as $\delta \to 0$ using the same arguments as in~\cite[Proposition A.1]{AMJ}. Below, we focus on the convergence of the jump terms $I_7$–$I_9$.\\
\smallskip
\noindent\textbf{Estimate for \boldmath$I_7$.}
Applying Kunita's first inequality~\cite[Theorem 4.4.23]{Applebaum}, we obtain
\begin{align}\label{eq:I7-Kunita}
&\mathbb{E} \sup_{t \in [0,T]} \left| \int_0^t \left\langle \int_{Z_1} \left[ \phi'(u^\delta(s)) G^\delta(s{-},z) - \phi'(u(s)) G(s{-},z) \right] \tilde N(ds,dz),\ \psi \right\rangle_{L^2(\mathcal{K})} \right|\nonumber\\
&\le C\,\mathbb{E} \left( \int_0^T \int_{Z_1} \left| \left\langle \phi'(u^\delta(s)) G^\delta(s,z) - \phi'(u(s)) G(s,z),\ \psi \right\rangle \right|^2 \nu(dz) ds \right)^{1/2}.
\end{align}
We  estimate the integrand using the inequality \(|A B - CD| \le |A - C||B| + |C||B - D|,\) which yields
\begin{align}
&\left| \left\langle \phi'(u^\delta(s)) G^\delta(s,z) - \phi'(u(s)) G(s,z),\ \psi \right\rangle_{L^2(\mathcal{K})} \right| \nonumber\\
&\le \left\| \phi'(u^\delta(s)) - \phi'(u(s)) \right\|_{L^2(\mathcal{K})} \cdot \left\| G^\delta(s,z) \right\|_{L^2(\mathcal{K})} \cdot \|\psi\|_{L^2(\mathcal{K})}\nonumber\\
&+ \left\| \phi'(u(s)) \right\|_{L^2(\mathcal{K})} \cdot \left\| G^\delta(s,z) - G(s,z) \right\|_{L^2(\mathcal{K})} \cdot \|\psi\|_{L^2(\mathcal{K})}. \notag
\end{align}
Substituting into~\eqref{eq:I7-Kunita} and using $\|\psi\|_{L^2(\mathcal{K})}$ bounded, we arrive at
\begin{align}\label{eq:I7-final}
\mathbb{E} \sup_{t \in [0,T]} |I_7(t)|
&\le C\,\mathbb{E} \left( \int_0^T \left\| \phi'(u^\delta(s)) - \phi'(u(s)) \right\|_{L^2(\mathcal{K})}^2
\int_{Z_1} \|G^\delta(s,z)\|_{L^2(\mathcal{K})}^2\, \nu(dz)\, ds \right)^{1/2} \notag\\
&\quad+ C\,\mathbb{E} \left( \int_0^T \| \phi'(u(s)) \|_{L^2(\mathcal{K})}^2
\int_{Z_1} \|G^\delta(s,z) - G(s,z)\|_{L^2(\mathcal{K})}^2\, \nu(dz)\, ds \right)^{1/2}.
\end{align}

\smallskip
\noindent Since $\phi'$ is Lipschitz, we have
\[
\left\| \phi'(u^\delta(s)) - \phi'(u(s)) \right\|_{L^2(\mathcal{K})} \longrightarrow 0
\quad \text{a.e. in } (s,\omega),
\]
and by assumption,
\[
\left\| G^\delta(s,z) - G(s,z) \right\|_{L^2(\mathcal{K})} \longrightarrow 0
\quad \text{a.e. in } (s,\omega,z).
\]
Moreover, we have the uniform bounds
\begin{align}\label{eq:I7-uniform}
&\mathbb{E} \int_0^T \left\| \phi'(u^\delta(s)) - \phi'(u(s)) \right\|_{L^2(\mathcal{K})}^2
\int_{Z_1} \|G^\delta(s,z)\|_{L^2(\mathcal{K})}^2 \nu(dz)\, ds \notag\\
&\quad+ \mathbb{E} \int_0^T \| \phi'(u(s)) \|_{L^2(\mathcal{K})}^2
\int_{Z_1} \|G^\delta(s,z) - G(s,z)\|_{L^2(\mathcal{K})}^2 \nu(dz)\, ds
\le C.
\end{align}
Hence, by the dominated convergence theorem, both terms on the right-hand side of~\eqref{eq:I7-final} vanish as $\delta \to 0$, so that $I_7 \to 0~   \text{in} ~L^1\big(\Omega; D([0,T])\big)$.

\smallskip\noindent
\textbf{Step 2: Convergence of $I_8$.}
This term is analogous to $I_7$ and treated identically using Kunita’s inequality for integrals against the non-compensated Poisson measure $N$.

\smallskip\noindent
\textbf{Step 3: Convergence of $I_9$.}
Note that
\begin{align}\label{ap-7}
&\left\langle \sum_{0\le s \le t} \left[ \phi(u^\delta(s)) - \phi(u^\delta(s^-)) - \phi'(u^\delta(s^-))(u^\delta(s)-u^\delta(s^-)) \right], \psi \right\rangle \nonumber\\
&\quad - \left\langle \sum_{0\le s \le t} \left[ \phi(u(s)) - \phi(u(s^-)) - \phi'(u(s^-))(u(s)-u(s^-)) \right], \psi \right\rangle  
 = \langle A_1 + A_2, \psi \rangle,
\end{align}
where
\begin{align*}
A_1 &= \sum_{0\le s\le t} \left[ \phi(u^\delta(s)) - \phi(u(s)) + \phi(u(s^-)) - \phi(u^\delta(s^-)) \right],\\
A_2 &= \sum_{0\le s\le t} \left[ \phi'(u^\delta(s^-))(u^\delta(s)-u^\delta(s^-)) - \phi'(u(s^-))(u(s)-u(s^-)) \right].
\end{align*}
Using again the Lipschitz continuity of $\phi$ and $\phi'$, together with the $L^2$ convergence of $u^\delta \to u$, we obtain that $A_1, A_2 \to 0$ in $L^2([0,T]\times\Omega)$, and hence also in probability. Therefore, \(I_9^\delta \to I_9, ~ \text{as} \to 0.\)

This completes the proof.

\end{proof}

{\bf  Data availability}   Data sharing is not applicable to this article as no datasets were generated or analyzed during the current study.

{\bf Conflict of interest}  The authors state that there is no conflict of interest.

\end{document}